\documentclass[11pt]{article}

\usepackage{geometry}
\usepackage[T1]{fontenc}
\usepackage[utf8]{inputenc}
\usepackage{microtype}
\usepackage{amsmath,amssymb,amsthm,mathtools}
\usepackage{bm}
\usepackage{algorithm}
\usepackage{algpseudocode}
\usepackage{booktabs}
\usepackage{multirow}
\usepackage{graphicx,graphics}
\graphicspath{{figures/}}
\usepackage{xcolor}
\usepackage{enumitem}
\usepackage{hyperref}

\hypersetup{
  colorlinks=true,
  linkcolor=blue,
  citecolor=blue,
  urlcolor=blue
}

\newcommand{\RR}{\mathbb{R}}

\newcommand{\vx}{\mathbf{x}}
\newcommand{\vy}{\mathbf{y}}
\newcommand{\vv}{\mathbf{v}}
\newcommand{\vp}{\mathbf{p}}
\newcommand{\vq}{\mathbf{q}}
\newcommand{\vu}{\mathbf{u}}
\newcommand{\vz}{\mathbf{z}}
\newcommand{\vg}{\mathbf{g}}
\newcommand{\vd}{\mathbf{d}}

\newcommand{\vone}{\mathbf{1}}
\newcommand{\vzero}{\mathbf{0}}
\newcommand{\cG}{\mathcal{G}}
\newcommand{\cV}{\mathcal{V}}
\newcommand{\cE}{\mathcal{E}}
\newcommand{\cN}{\mathcal{N}}
\newcommand{\cU}{\mathcal{U}}
\newcommand{\cC}{\mathcal{C}}
\newcommand{\cX}{\mathcal{X}}
\newcommand{\cY}{\mathcal{Y}}
\newcommand{\cL}{\mathcal{L}}

\newcommand{\sign}{\operatorname{sign}}
\newcommand{\grad}{\nabla}
\newcommand{\norm}[1]{\left\lVert #1 \right\rVert}

\newcommand{\col}{\operatorname{col}}
\newcommand{\dist}{\operatorname{dist}}

\newtheorem{theorem}{Theorem}[section]
\newtheorem{lemma}[theorem]{Lemma}

\newtheorem{corollary}[theorem]{Corollary}

\newtheorem{assumption}{Assumption}

\numberwithin{equation}{section}

\begin{document}

\title{Exact Decentralized Optimization via Explicit $\ell_1$ Consensus Penalties}

\author{Hong Wang\thanks{School of Artificial Intelligence, Shenzhen Technology University, 3002 
Lantian Road, Pingshan District, Shenzhen Guangdong, China, 518118. Email: hitwanghong@163.com}}
\date{}

\maketitle

\begin{abstract}
Consensus optimization enables autonomous agents to solve joint tasks through peer-to-peer exchanges alone. Classical decentralized gradient descent is appealing for its minimal state but fails to achieve exact consensus with fixed stepsizes unless additional trackers or dual variables are introduced. We revisit penalty methods and introduce a decentralized two-layer framework that couples an outer penalty-continuation loop with an inner plug-and-play saddle-point solver. Any primal-dual routine that satisfies simple stationarity and communication conditions can be used; when instantiated with a proximal-gradient solver, the framework yields the DP$^2$G algorithm, which reaches exact consensus with constant stepsizes, stores only one dual residual per agent, and requires exactly two short message exchanges per inner iteration. An explicit $\ell_1$ penalty enforces agreement and, once above a computable threshold, makes the penalized and constrained problems equivalent. Leveraging the Kurdyka-\L{}ojasiewicz property, we prove global convergence, vanishing disagreement, and linear rates for strongly convex objectives under 
any admissible inner solver. Experiments on distributed least squares, logistic regression, and elastic-net tasks across various networks demonstrate that DP$^2$G outperforms DGD-type methods in both convergence speed and communication efficiency, is competitive with gradient-tracking approaches while using less memory, and naturally accommodates composite objectives.
\end{abstract}

\section{Introduction}
Consensus optimization is central to modern networked systems including distributed sensing, distributed learning, smart grids, and cooperative robotics. In these settings, $n$ autonomous agents aim to solve
\[
  f(x) = \frac{1}{n}\sum_{i=1}^n f_i(x)
\]
using only local computation and peer-to-peer communication. Practical deployments must satisfy three core requirements: minimal per-agent memory, low communication overhead, and robust convergence under realistic networking constraints~\cite{nedic2009distributed,yuan2016dgd,yang2021distributed,wang2023distributed}. Failing to meet these constraints increases implementation costs and undermines system resilience.

Classical decentralized gradient descent (DGD) meets low-memory requirements but achieves exact consensus only with vanishing stepsizes~\cite{nedic2009distributed,yuan2016dgd}. Such diminishing stepsizes degrade practical performance and complicate parameter tuning, while fixed stepsizes result in persistent disagreement. Gradient-tracking and primal-dual schemes~\cite{shi2015extra,qu2018,li2021nids} recover exactness with constant steps by equipping each node with auxiliary states (gradient trackers, dual multipliers, or both), thereby doubling memory usage and increasing communication overhead. These drawbacks motivate algorithms that preserve the simplicity of DGD while delivering the accuracy of more advanced methods.

Penalty methods provide that bridge. Quadratic penalties and augmented Lagrangians are analytically convenient but require large penalty weights and full dual updates, both of which erode efficiency in bandwidth-limited settings. Explicit $\ell_1$ penalties instead enjoy an exact-penalty property: once the coefficient exceeds a computable threshold, the penalized and constrained formulations share the same solutions. The main technical challenges arise from managing the resulting nonseparable disagreement term and coordinating the penalty schedule without centralized control, particularly when aiming for convergence guarantees beyond convex objectives.

We address these challenges using a modular two-layer framework. The outer layer runs a fully decentralized penalty-continuation scheme that keeps each agent's state to one primal vector and one dual residual, matching the memory footprint of DGD while steering the penalty toward the exactness threshold. The inner layer is a plug-and-play saddle-point solver that must guarantee a verifiable decrease in a composite optimality residual; our theory is independent of the specific routine provided it meets simple stationarity, communication, and warm-start requirements. Instantiating the inner layer with a primal-dual proximal-gradient method yields the DP$^2$G algorithm evaluated in our experiments, yet the same interface accommodates Chambolle--Pock variants, accelerated saddle-point updates, or inexact ADMM iterations. Once the penalty saturates, the framework attains the accuracy of constrained consensus methods without sacrificing the simplicity of DGD.

\paragraph{Contributions.} The paper makes three main contributions.
\begin{itemize}[leftmargin=1.6em]
  \item \textbf{Exact-penalty perspective.} We recast consensus optimization by penalizing disagreement through the operator $Z=(I-W)\otimes I_m$ with an explicit $\ell_1$ term and, using Hoffman's bound, derive a computable threshold beyond which the penalized and constrained formulations are equivalent.
  \item \textbf{Modular two-layer framework.} We formalize the penalty-continuation outer loop and the plug-and-play inner saddle solver, specifying the interface in terms of residual reduction, communication accounting, and warm-start requirements. When the inner routine is a primal-dual proximal-gradient method, we obtain DP$^2$G, which stores a single primal and dual vector per agent and needs only two neighbor exchanges per inner iteration.
  \item \textbf{Global guarantees and evidence.} Leveraging the Kurdyka-\L{}ojasiewicz (K\L{}) property of the penalized objective, we adapt the Lyapunov analysis of preconditioned primal-dual gradient methods~\cite{guo2023ppdg} to the distributed setting and prove global convergence to consensual critical points, with linear rates under strong convexity for \emph{any} admissible inner solver. The same analysis covers semi-algebraic nonconvex objectives, and our experiments on least squares, logistic regression, and elastic-net tasks across several topologies corroborate the theoretical predictions.
\end{itemize}

\subsection{Related Work and Recent Advances}
\paragraph{Gradient tracking.} EXTRA~\cite{shi2015extra} pioneered gradient-correction steps that cancel consensus bias and allow fixed stepsizes. Subsequent work generalized the idea to directed graphs (Push-DIGing~\cite{qu2018}), introduced Nesterov-type acceleration (NIDS~\cite{li2021nids}), and analyzed unbalanced graphs or nonconvex objectives~\cite{xu2023accelerated,zhang2021gradient}. Recent efforts pursue optimal gradient complexity~\cite{li2024optimal}, modular frameworks that mix local computation with tracking~\cite{shen2025gradient,sun2025flexible}, and variance-reduced stochastic variants for large-scale learning~\cite{li2023variance,chen2021decentralized,wang2024privsgp,mei2024variance}. All such methods keep auxiliary gradient trackers in memory and in flight, which doubles the storage relative to DGD and motivates our search for tracker-free yet exact algorithms.

\paragraph{ADMM and primal-dual algorithms.} ADMM handles consensus by introducing dual variables and augmented penalties~\cite{boyd2011admm}. It is robust but often slower than first-order trackers on smooth problems~\cite{shi2015extra}, motivating communication-efficient refinements based on compression, event-triggered communication, or adaptive penalty updates~\cite{wu2023compressed,chen2017projected,xu2022adaptive}. Linear rates are available under strong convexity~\cite{jakovetic2014linear,chang2015multi}, albeit at the cost of storing both primal and dual vectors per agent. Our framework borrows the primal-dual perspective but decouples the penalty schedule (outer layer) from the particular saddle solver (inner layer), allowing lighter per-node states when DP$^2$G is selected.

\paragraph{Penalty methods.} Classical quadratic penalties and augmented Lagrangians can degrade conditioning or require full dual updates, whereas nonquadratic penalties such as $\ell_1$ and elastic-net terms offer robustness to noise and can encourage sparse disagreement corrections~\cite{liu2021penalty,reisizadeh2020quantized,shi2022communication}. Recent studies leveraged penalties for constrained consensus~\cite{carrillo2023constrained} or bilevel formulations~\cite{chen2024penalty}, yet most rely on diminishing penalty sequences or heuristics without exact-penalty guarantees. Our contribution is to couple explicit $\ell_1$ continuation with a plug-and-play saddle solver, yielding fixed-stepsize guarantees while keeping the memory footprint comparable to DGD.

\subsection{Notation}
Bold lowercase letters (e.g., $\vx_i$) denote vectors and uppercase letters (e.g., $W$) denote matrices. The operator $\mathrm{col}(\cdot)$ stacks its arguments column-wise; $\vone$ and $\vzero$ are all-ones and all-zeros vectors of compatible dimensions. The Kronecker product is written $\otimes$, and $\RR^m$ denotes the $m$-dimensional Euclidean space equipped with norm $\|\cdot\|$. We reserve $\|\cdot\|_1$ for the $\ell_1$ norm. For a closed set $\cX$, $\operatorname{dist}(x,\cX)$ is the Euclidean distance from $x$ to $\cX$, and $\mathrm{Proj}_{\cX}(x)$ is the Euclidean projection. The limiting subdifferential of a function $g$ at $x$ is $\partial g(x)$. Agent indices lie in $[n]=\{1,2,\ldots,n\}$, the communication graph is $\cG=(\cV,\cE)$, and $\cN_i$ denotes the neighbor set of agent $i$.

\subsection{Organization}
The remaining content is organized as follows.
Section~\ref{sec:setup} formalizes the consensus model and the explicit $\ell_1$ penalty reformulation. Section~\ref{sec:algorithm} details the modular two-layer framework together with its DP$^2$G instantiation. Section~\ref{sec:analysis} establishes convergence guarantees, and Section~\ref{sec:experiments} reports the numerical results. Section~\ref{sec:conclusion} summarizes the main findings and outlines future directions.

\section{Consensus Model and Penalty Reformulation}\label{sec:setup}
\subsection{Network Topology}
We adopt the standard modeling framework used in decentralized optimization. Let $\cG=(\cV,\cE)$ 
be a connected undirected graph with node set $\cV=[n]$ and edge set $\cE$. Agent $i$ exchanges 
information only with neighbors in $\cN_i=\{j:(i,j)\in\cE\}$. Communication is governed by a mixing 
matrix $W$ that respects the sparsity of $\cG$ and satisfies the conditions in 
Assumption~\ref{ass:mixing-mat}.
\begin{assumption} \label{ass:mixing-mat}
The mixing matrix $W=[w_{ij}] \in \mathbb{R}^{n \times n}$ satisfies the following conditions:
\begin{enumerate}
\item For any $i \neq j$, if $(i,j) \notin \cE$, then $w_{ij} =0$, and $w_{ij} > 0$ otherwise.
\item $W$ is symmetric, namely, $W = W^{\mathsf{T}}$.
\item $W$ is doubly stochastic, i.e., $W \vone = \vone$ and $\vone^{\mathsf{T}} W = 
\vone^{\mathsf{T}}$.
\item $-I \prec W \preceq  I$.
\end{enumerate}
\end{assumption}

Denote the eigenvalues of $W$ by $\lambda_1(W)\geq\cdots\geq\lambda_n(W)$. 
Assumption~\ref{ass:mixing-mat} guarantees 
$-1<\lambda_n(W)\leq\cdots\leq\lambda_2(W)<\lambda_1(W)=1$. The spectral gap $1-\zeta$, with 
$\zeta=\max\{|\lambda_2(W)|,|\lambda_n(W)|\}$, is therefore strictly positive and measures how 
quickly consensus information diffuses through the network.

\subsection{Consensus Optimization Problem}
Each agent privately holds a differentiable function $f_i:\RR^m\rightarrow\RR$. The consensus 
optimization problem seeks
\begin{equation}\label{eq:consensus}
  \min_{x\in\RR^m} \; f(x) = \sum_{i=1}^n f_i(x),
\end{equation}
where the factor $1/n$ is omitted because it does not affect optimal solutions.

Introducing local copies $\vx_i\in\RR^m$ and stacking $\vx=\col(\vx_1,\ldots,\vx_n)$ yields the 
equivalent constrained formulation
\begin{equation}\label{eq:constrained}
  \min_{\vx\in\RR^{nm}} \; F(\vx) = \sum_{i=1}^n f_i(\vx_i) \quad\text{s.t.}\quad \vx_1=\cdots=\vx_n.
\end{equation}
Let $Z=(I_n-W)\otimes I_m$. Because $W$ is symmetric and doubly stochastic, $I_n-W$ is symmetric and positive semidefinite with nullspace spanned by $\vone$. Consequently,
\begin{equation*}
Z^{\mathsf{T}} = Z, 
\qquad
Z(\vone \otimes x) = \bigl((I_n-W)\vone\bigr)\otimes x = \vzero
\quad \forall x\in\RR^m.
\end{equation*}

Since $I_n-W\succeq 0$, it follows that $Z=(I_n-W)\otimes I_m\succeq 0$, and its nullspace is $\{\vone\otimes v : v\in\RR^m\}$, 
corresponding exactly to the consensus subspace where all agents agree.

\begin{assumption} \label{ass:smooth}
For each $i\in[n]$, the function $f_i$ is proper, closed, bounded below, and continuously differentiable 
with $L_i$-Lipschitz continuous gradient. Define $L_{\max}=\max_i \{L_i\}$.
\end{assumption}

\subsection{The $\ell_1$ Consensus Penalty}
We study the penalized objective
\begin{equation}\label{eq:penalty}
  \Phi_\rho(\vx) = F(\vx) + \rho \norm{Z\vx}_1 ,
\end{equation}
where $\rho>0$ is the penalty parameter and $Z = (I-W) \otimes I_m$. Denoting $\vu = Z \vx$ and
\begin{equation}\label{eq:row-residual}
  \vu_i := (Z\vx)_i = (1-w_{ii})\vx_i - \sum_{j\in \cN_i} w_{ij} \vx_j,
\end{equation}
we have $\|Z\vx\|_1 = \sum_{i=1}^n \|\vu_i\|_1$. Importantly, each component of $\vu_i$ depends on
both $\vx_i$ and the neighbor variables $\vx_j$, so the penalty is 
\emph{not} separable across agents. Any update that manipulates a single block $\vx_i$ must account for
the contribution of $\vx_i$ to $\vu_i$ and to the rows $\vu_r$ of its in-neighbors.

\begin{lemma}[Local structure of the penalty subgradient]\label{lem:subgrad-structure}
Let $g_i(\vx) \in (\partial_{\vx} \|Z\vx\|_1)_i$ be a subgradient of the penalty with respect to agent $i$. Then
\begin{equation}\label{eq:subgrad-structure}
  g_i(\vx) = (1-w_{ii})\sign(\vu_i) - \sum_{r: i \in \cN_r} w_{r i} \sign(\vu_r),
\end{equation}
where the $\sign(\cdot)$ operator acts component-wise and produces subgradients in $[-1,1]$ when the 
corresponding residual coordinate is zero.
\end{lemma}

\begin{proof}
Recall that $\|Z\vx\|_1 = \sum_{r=1}^n \|\vu_r\|_1$ with $\vu_r = (1-w_{rr})\vx_r - \sum_{j\in\cN_r} w_{rj}\vx_j$. 
Fix $i$ and perturb only the block $\vx_i$ by $h\in\RR^m$. The resulting change in the penalty is
\begin{align*}
  \|Z(\vx + e_i \otimes h)\|_1 - \|Z\vx\|_1 
  &= \Big\|\vu_i + (1-w_{ii})h\Big\|_1 - \|\vu_i\|_1 \\
  &\quad + \sum_{r: i\in\cN_r} \Big( \big\|\vu_r - w_{ri}h\big\|_1 - \|\vu_r\|_1 \Big),
\end{align*}
where $e_i$ is the $i$th unit vector in $\RR^n$. Each term in the sum is convex in $h$ and admits the 
subgradient representation
\[
  \partial_h \|\vu_i + (1-w_{ii})h\|_1 = (1-w_{ii})\sign(\vu_i),\qquad
  \partial_h \|\vu_r - w_{ri}h\|_1 = - w_{ri} \sign(\vu_r),
\]
with the convention that $\sign(0)$ is the interval $[-1,1]$ applied component-wise. Summing these 
contributions yields precisely~\eqref{eq:subgrad-structure}. Because the limiting subdifferential of a 
finite sum of convex functions is the sum of the individual limiting subdifferentials, the stated expression 
describes every element of $(\partial_{\vx} \|Z\vx\|_1)_i$.
\end{proof}

\begin{assumption}[Level-boundedness]\label{ass:level}
The objective function of \eqref{eq:constrained} $F$ has bounded level sets. Equivalently, the penalized objective function $\Phi_\rho$ has bounded level sets.
\end{assumption}

\begin{theorem}[Exactness of $\ell_1$ penalty]\label{thm:exactness}
Suppose Assumption~\ref{ass:smooth} holds and~\eqref{eq:consensus} has a nonempty solution set 
$\cX^\ast$. Then there exists $\bar{\rho}>0$ such that for any $\rho\geq\bar{\rho}$:
\begin{enumerate}[label=(\alph*),leftmargin=1.6em]
  \item Every consensual optimal point $\vx^\ast=\vone\otimes x^\ast$ with $x^\ast\in\cX^\ast$ is a 
  local minimizer of $\Phi_\rho$.
  \item If $\vx^\dagger$ is a local minimizer of $\Phi_\rho$, then $Z\vx^\dagger=\vzero$ and its 
  consensus component $x^\dagger$ solves~\eqref{eq:consensus}.
\end{enumerate}
\end{theorem}

\begin{proof}
Let $\cC=\mathrm{null}(Z)=\{\vone\otimes x : x\in\RR^m\}$ and fix any consensual solution $x^\ast\in\cX^\ast$. 
Set $\vx^\ast=\vone\otimes x^\ast$, so that $\vx^\ast\in \cC$ and $Z\vx^\ast=\vzero$. Because each $f_i$ is 
continuously differentiable, $F$ is differentiable and therefore locally Lipschitz. Let $L_F>0$ denote a 
Lipschitz constant of $F$ on a compact neighborhood $\mathcal{V}$ of $\vx^\ast$ whose existence follows 
from continuity.

Hoffman's bound for the linear system $Z\vx=0$ \cite{hoffman1952approx} yields a constant $C>0$ such 
that
\begin{equation}\label{eq:hoffman}
  \dist(\vx,\cC) \le C\,\|Z\vx\|_1 \qquad \forall\, \vx\in\RR^{nm}.
\end{equation}

\emph{Part (a).} For any $\vx\in\mathcal{V}$ we use the triangle inequality and the Lipschitz property to 
write $F(\vx) \ge F(\vx^\ast) - L_F \dist(\vx,\cC)$. Combining this inequality with~\eqref{eq:hoffman} we 
obtain
\[
  \Phi_\rho(\vx) = F(\vx) + \rho\|Z\vx\|_1 \ge F(\vx^\ast) + (\rho - C L_F)\|Z\vx\|_1.
\]
Whenever $\rho \ge C L_F$, the right-hand side is minimized at $\|Z\vx\|_1=0$, i.e., at $\vx\in\cC$. Because 
$\vx^\ast$ belongs to $\cC$ and is optimal for~\eqref{eq:constrained}, we conclude that $\vx^\ast$ is a local 
minimizer of $\Phi_\rho$ for every $\rho \ge C L_F$. The constant $\bar\rho=C L_F$ therefore satisfies 
statement (a).

\emph{Part (b).} Let $\vx^\dagger$ be any local minimizer of $\Phi_\rho$ with $\rho>\bar\rho$. Pick a compact 
neighborhood $\mathcal{W}$ of $\vx^\dagger$ on which $F$ is Lipschitz with constant $L_{\mathcal{W}}$. Take 
$\hat{\vx}$ to be the Euclidean projection of $\vx^\dagger$ onto $\cC$. Then $Z\hat{\vx}=\vzero$ and 
$\hat{\vx}\in \mathcal{W}$ for $\mathcal{W}$ small enough. Hence
\begin{align}
  F(\hat{\vx}) &\le F(\vx^\dagger) + L_{\mathcal{W}}\,\|\hat{\vx}-\vx^\dagger\| 
  = F(\vx^\dagger) + L_{\mathcal{W}}\,\dist(\vx^\dagger,\cC) \\
  &\le F(\vx^\dagger) + L_{\mathcal{W}} C \|Z\vx^\dagger\|_1,
  \label{eq:proj-lip}
\end{align}
where the last step uses~\eqref{eq:hoffman}. Optimality of $\vx^\dagger$ on $\mathcal{W}$ yields
\begin{equation}
  F(\vx^\dagger) + \rho\|Z\vx^\dagger\|_1 \le F(\hat{\vx}) + \rho\|Z\hat{\vx}\|_1 = F(\hat{\vx}).
  \label{eq:local-opt}
\end{equation}
Substituting~\eqref{eq:proj-lip} into~\eqref{eq:local-opt} gives $(\rho - C L_{\mathcal{W}})\|Z\vx^\dagger\|_1 \le 0$. 
As $\rho>\bar\rho\ge C L_{\mathcal{W}}$, the only possibility is $Z\vx^\dagger=\vzero$, showing that every 
local minimizer is consensual. Finally, when $\vx^\dagger\in\cC$ the penalty term vanishes and
$F(\vx^\dagger) \le F(\hat{\vx})$ for every $\hat{\vx}\in \cC$ close to $\vx^\dagger$, implying that the shared 
component solves~\eqref{eq:consensus}.
\end{proof}

\section{Decentralized Primal-Dual Proximal Gradient Algorithm}\label{sec:algorithm}
We now derive the DP$^2$G update and formalize the resulting two-layer architecture. The outer layer adapts the penalty parameter $\rho_k$ at iteration $k$, whereas the inner layer performs primal-dual proximal-gradient steps for a fixed penalty until a verifiable optimality condition is satisfied.

\subsection{Primal-dual splitting}
Introduce the saddle formulation of the penalized problem
\begin{equation}\label{eq:saddle}
  \min_{\vx \in \RR^{nm}} \max_{\vy \in \cY_\rho} \; \cL(\vx,\vy;\rho) 
  = F(\vx) + \langle Z\vx, \vy \rangle - \delta_{\cY_\rho}(\vy),
\end{equation}
where $\cY_\rho = \{ \vy \in \RR^{nm} : \|\vy\|_\infty \le \rho \}$ and $\delta_{\cY_\rho}$ is the indicator of the hypercube. Eliminating $\vy$ recovers $\Phi_\rho(\vx)$ because the conjugate of $\delta_{\cY_\rho}$ is the norm $\rho \|\cdot\|_1$. The gradient of $F$ is block-separable, while the adjoint $\vv := Z^{\mathsf{T}}\vy$ inherits the sparsity of the mixing matrix:
\begin{equation}\label{eq:dual-adjoint}
  \vv_i := (Z^{\mathsf{T}}\vy)_i = (1-w_{ii})\vy_i - \sum_{j\in \cN_i} w_{ji}\vy_j.
\end{equation}
Hence each agent only needs the dual variables of its neighbors.

Applying the primal-dual hybrid gradient scheme with the recommended extrapolation parameter $\theta=1$ yields
\begin{subequations}\label{eq:pd-updates}
\begin{align}
  \vy^{t+1} &= \mathrm{Proj}_{\cY_{\rho_k}}\bigl( \vy^{t} + \sigma Z \bar{\vx}^{t} \bigr), \label{eq:dual-update}\\
  \vx^{t+1} &= \vx^{t} - \alpha \bigl( \nabla F(\vx^{t}) + Z^{\mathsf{T}}\vy^{t+1} \bigr), \label{eq:primal-update}\\
  \bar{\vx}^{t+1} &= \vx^{t+1} + \bigl(\vx^{t+1}-\vx^{t}\bigr), \label{eq:extrapolate}
\end{align}
\end{subequations}
where $\bar{\vx}^0=\vx^0$. The projection $\mathrm{Proj}_{\cY_{\rho_k}}$ acts as a soft threshold on the disagreement residual: whenever a component of $Z\bar{\vx}^{t}$ exceeds $\rho_k$, the corresponding dual variable saturates at $\pm\rho_k$ and memorizes the sign of the disagreement. The fixed steps obey $0<\alpha < 1/L_{\max}$ and $0<\sigma < 1/(\alpha(1-\lambda_n(W))^2)$.

The local form required by agent $i$ becomes
\begin{subequations}\label{eq:local-updates}
\begin{align}
  \vy_i^{t+1} &= \mathrm{Proj}_{[-\rho_k,\rho_k]}\Bigl( \vy_i^{t} + \sigma \bar{\vu}_i^{t} \Bigr), \\
  \vx_i^{t+1} &= \vx_i^{t} - \alpha \bigl( \nabla f_i(\vx_i^{t}) + \vv_i^{t+1} \bigr), \\
  \bar{\vx}_i^{t+1} &= \vx_i^{t+1} + \bigl(\vx_i^{t+1}-\vx_i^{t}\bigr),
\end{align}
\end{subequations}
where $\bar{\vu}_i^{t}=(Z\bar{\vx}^{t})_i$ is computed from extrapolated neighbor information. Each inner iteration therefore still requires two neighbor exchanges: one to form $\bar{\vu}_i^{t}$ and one to share the fresh dual variables $\vy_i^{t+1}$ before evaluating $\vv_i^{t+1}=(Z^{\mathsf{T}}\vy^{t+1})_i$.

\subsection{Two-layer adaptive architecture}
The outer layer increases the penalty parameter until the exactness threshold from Theorem~\ref{thm:exactness} is exceeded, while the inner layer executes~\eqref{eq:local-updates} until a prescribed optimality tolerance is met. Communication-wise, DP$^2$G needs two neighbor exchanges per inner iteration (one for $\vx$, one for $\vy$) and only stores $\vx_i$ and $\vy_i$ locally. The full procedure is summarized in Algorithm~\ref{alg:dp2g}.

\begin{algorithm}[t]
  \caption{Two-layer Decentralized Primal-Dual Proximal Gradient (DP$^2$G)}
  \label{alg:dp2g}
  \begin{algorithmic}[1]
    \State \textbf{Input:} Initial $\vx_i^0$, duals $\vy_i^0=\vzero$, penalty $\rho_0>0$, primal step $\alpha>0$, dual step $\sigma>0$, growth factor $\beta>1$, cap $\rho_{\max}$, tolerances $\{\varepsilon_k\}$ and $\{\delta_k\}$.
    \State Set $k=0$ and $\rho_k=\rho_0$.
    \While{not terminated}
      \State Set $t = 0$ and $\bar{\vx}_i^0 = \vx_i^{k}$.
      \Repeat \Comment{Inner primal-dual iterations with fixed $\rho_k$}
        \State Exchange $\bar{\vx}_i^{t}$ with neighbors and compute $\bar{\vu}_i^{t}$ according to \eqref{eq:row-residual}.
        \State Dual update: $\vy_i^{t+1} = \mathrm{Proj}_{[-\rho_k,\rho_k]}( \vy_i^{t} + \sigma \bar{\vu}_i^{t} )$.
        \State Exchange $\vy_i^{t+1}$ with neighbors and compute $\vv_i^{t+1}$ according to \eqref{eq:dual-adjoint}.
        \State Primal update: $\vx_i^{t+1} = \vx_i^{t} - \alpha ( \nabla f_i(\vx_i^{t}) + \vv_i^{t+1} )$.
        \State Extrapolation: $\bar{\vx}_i^{t+1} = \vx_i^{t+1} + (\vx_i^{t+1}-\vx_i^{t})$.
        \State $t \leftarrow t+1$.
      \Until $\bigl\| \nabla f_i(\vx_i^{t}) + \vv_i^{t} \bigr\| \le \varepsilon_k$ for all $i$. \label{alg:stop-rule-inner}
      \State Set $\vx_i^{k+1} = \vx_i^{t}$, $\vy_i^{k+1} = \vy_i^{t}$.
      \State Exchange the fresh $\vx_i^{k+1}$ and compute $d_i^{k+1} = \big\| \vu_i^{k+1}\big\|_1$. 
      \If{$d_i^{k+1} \le \delta_k$ for all $i$ \label{alg:stop-rule-outer}} 
        \State \textbf{break}
      \Else
        \State Update penalty: $\rho_{k+1}=\min\{\beta \rho_k, \rho_{\max}\}$, and $k \leftarrow k+1$.
      \EndIf
    \EndWhile
  \end{algorithmic}
\end{algorithm}

\begin{assumption}[Penalty schedule and stepsize]\label{ass:penalty}
The penalty sequence satisfies $\rho_0>0$, $\rho_{\max}<\infty$, and $\rho_{k+1}=\min\{\beta\rho_k,\rho_{\max}\}$ with $\beta>1$. The cap $\rho_{\max}$ is chosen so that $\rho_{\max} \ge \bar{\rho}$ (the exactness threshold from Theorem~\ref{thm:exactness}) and $\rho_{\max} > G$, where $G$ is any known upper bound on $\max_{i\in[n]}\|\nabla f_i(x)\|$ over the sublevel set $\{\vx: \Phi_{\rho_0}(\vx)\le \Phi_{\rho_0}(\vx^0)\}$ guaranteed by Assumption~\ref{ass:level}. The fixed stepsizes obey $0<\alpha<1/(3L_{\max})$ and $0<\sigma<1/(\alpha(1-\lambda_n(W))^2)$.
\end{assumption}

\begin{assumption}[Tolerance sequences]\label{ass:tolerances}
The stationarity tolerances $\{\varepsilon_k\}$ and consensus tolerances $\{\delta_k\}$ satisfy $\varepsilon_k > 0$ and $\delta_k > 0$, where both sequences are nonincreasing and satisfy $\varepsilon_k \to 0$ and $\delta_k \to 0$ as $k \to \infty$. Typical choices include $\varepsilon_k = \varepsilon_0 / k^{\eta_1}$ and $\delta_k = \delta_0 / k^{\eta_2}$ for $\eta_1 > 0$ and $\eta_2 > 0$, or exponentially decaying rules such as $\varepsilon_k = \varepsilon_0 \theta_1^{k}$ and $\delta_k = \delta_0 \theta_2^{k}$ with $\theta_1, \theta_2 \in (0,1)$. In practice, choosing $\delta_k$ to decay faster than $\varepsilon_k$ (e.g., $\eta_2 = 2\eta_1$) ensures consensus is achieved before final stationarity.
\end{assumption}

The stepsize bounds $\alpha < 1/(3L_{\max})$ and $\sigma < 1/(\alpha\kappa_Z^2)$ with $\kappa_Z = \|Z\|_2 = 1-\lambda_n(W)$ are stricter than those for proximal or gradient descent with Lipschitz gradients. These restrictions ensure sufficient descent in the Lyapunov analysis (see Lemma~\ref{lem:lyapunov-descent}). The inner loop performs proximal gradient steps until the optimality condition is satisfied, while the growth factor $\beta > 1$ adaptively strengthens the penalty when progress slows.

\subsection{Stationarity certificates}
The inner stopping criterion relies on the distance of the composite subgradient to zero. Because the dual variables belong to $\cY_{\rho_k}$, the penalty subgradient is readily available through $Z^{\mathsf{T}}\vy$.

\begin{lemma}[Stationarity residual]\label{lem:stationarity-residual}
Let $(\vx,\vy)$ satisfy $\vy \in \rho_k \partial \|Z\vx\|_1$. Then
\begin{equation}\label{eq:stationarity-residual}
  \dist\bigl( \vzero, \nabla F(\vx) + \partial (\rho_k \|Z\cdot\|_1)(\vx) \bigr) =
  \bigl\| \nabla F(\vx) + Z^{\mathsf{T}}\vy \bigr\|.
\end{equation}
For a general $\vy\in\cY_{\rho_k}$ the right-hand side provides a computable upper bound,
\[
  \dist\bigl( \vzero, \nabla F(\vx) + \partial (\rho_k \|Z\cdot\|_1)(\vx) \bigr)
  \le \bigl\| \nabla F(\vx) + Z^{\mathsf{T}}\vy \bigr\|.
\]
The inequality is strict whenever $\vy$ lies in the interior of $\cY_{\rho_k}$, because then $Z^{\mathsf{T}}\vy$ does not belong to $\partial (\rho_k \|Z\cdot\|_1)(\vx)$.
\end{lemma}

\begin{proof}
Because the conjugate of $\delta_{\cY_{\rho_k}}$ is $\rho_k\|\cdot\|_1$, the subdifferential of the penalty reads
\[
  \partial (\rho_k \|Z\cdot\|_1)(\vx) = Z^{\mathsf{T}}\!\bigl(\rho_k \partial \|Z\vx\|_1\bigr).
\]
Hence $\vy \in \rho_k \partial \|Z\vx\|_1$ implies $Z^{\mathsf{T}}\vy \in \partial (\rho_k \|Z\cdot\|_1)(\vx)$ and therefore attains the minimum distance from $-\nabla F(\vx)$ to the set $\partial (\rho_k\|Z\cdot\|_1)(\vx)$, which is exactly~\eqref{eq:stationarity-residual}. Furthermore, $\rho_k \partial \|Z\vx\|_1$ is a nonempty closed convex subset of $\cY_{\rho_k}$. Let $\widehat{\vy}$ be the Euclidean projection of an arbitrary $\vy\in\cY_{\rho_k}$ onto this subset. Then $\widehat{\vy}\in\rho_k \partial \|Z\vx\|_1$ and $\|\nabla F(\vx) + Z^{\mathsf{T}}\widehat{\vy}\| \le \|\nabla F(\vx) + Z^{\mathsf{T}}\vy\|$, which establishes the stated upper bound for general $\vy\in\cY_{\rho_k}$.
\end{proof}

Lemma~\ref{lem:stationarity-residual} justifies the inner-loop test in Algorithm~\ref{alg:dp2g}: each agent only needs the local block of $\nabla F(\vx) + Z^{\mathsf{T}}\vy$, which is exactly the quantity already computed during the primal update.

\section{Convergence Analysis}\label{sec:analysis}
We now prove that DP$^2$G converges globally under Assumptions~\ref{ass:mixing-mat}--\ref{ass:tolerances}. The analysis mirrors the Lyapunov technique developed for the preconditioned primal-dual gradient (PPDG) method in~\cite{guo2023ppdg}: we first study the inner primal-dual iterations for a fixed penalty, show that they generate a finite-length trajectory by exploiting a carefully constructed Lyapunov function, and then invoke the Kurdyka-\L{}ojasiewicz (K\L{}) property to pass from subsequence convergence to convergence of the whole sequence. The outer penalty updates only appear at the very end of the argument.

\paragraph{Kurdyka-\L{}ojasiewicz framework.}
The K\L{} property quantifies how sharply a function grows around its critical points. A proper lower semicontinuous function $\phi:\RR^d\to\RR\cup\{+\infty\}$ has the K\L{} property at $x^\star\in\operatorname{dom}\partial\phi$ if there exist $\eta>0$, a neighbourhood $\cU$ of $x^\star$, and a concave continuous desingularizing function $\varphi:[0,\eta)\to\RR_+$ that satisfies $\varphi(0)=0$, $\varphi'(s)>0$ for $s>0$, and
\[
  \varphi'\bigl(\phi(x)-\phi(x^\star)\bigr)\dist\bigl(0,\partial\phi(x)\bigr)\ge 1
  \qquad\forall x\in\cU:\; 0<\phi(x)-\phi(x^\star)<\eta .
\]
Semi-algebraic functions satisfy the K\L{} property globally~\cite{attouch2013}, and so does $\Phi_\rho$ because it is the sum of a smooth semi-algebraic function and the polyhedral norm $\rho\|Z\cdot\|_1$. Throughout this section we exploit the K\L{} property only after establishing boundedness and finite-length behaviour of the inner loop trajectories, in line with~\cite{guo2023ppdg}.

\subsection{Lyapunov analysis for the inner loop}
Recall from~\eqref{eq:saddle} that the penalized saddle formulation is driven by the Lagrangian $\cL(\vx,\vy;\rho)=F(\vx)+\langle Z\vx,\vy\rangle-\delta_{\cY_\rho}(\vy)$. For brevity we write $\cL_\rho(\vx,\vy)=\cL(\vx,\vy;\rho)$,
and let the disagreement operator norm be $\kappa_Z = \|Z\|_2 = 1-\lambda_n(W)$. Assumption~\ref{ass:penalty} guarantees step sizes $0<\alpha<1/(3L_{\max})$ and $0<\sigma < 1/(\alpha\kappa_Z^2)$. Motivated by~\cite{guo2023ppdg}, we augment $\cL_\rho$ with squared-difference terms,
\begin{equation}\label{eq:lyapunov-def}
  \Psi_\rho(\vx,\vy,\vp,\vq)
  := \cL_\rho(\vx,\vy) - a\|\vx-\vp\|^2 + b\|\vx-\vq\|^2, \quad \forall \vx, \vp, \vq \in \RR^{nm}, \; \vy \in \cY_{\rho}
\end{equation}
where $a$ and $b$ follow the construction in~\cite[(2.9)]{guo2023ppdg} with a tuning parameter $\delta\in(0,1/5]$; explicitly,
\[
  a = \frac{\delta}{\alpha},
  \qquad
  b = \frac{1}{2\alpha} - \frac{\delta}{\alpha} - \frac{L_{\max}}{4} - \delta L_{\max}
      - \frac{\alpha \delta L_{\max}^2}{2} + \frac{\alpha L_{\max}^2}{4\delta}.
\]
These choices ensure $a,b>0$ whenever $\alpha<1/(3L_{\max})$. We further set $c = b - \tfrac{\alpha L_{\max}^2}{2\delta}>0$. The four arguments in~\eqref{eq:lyapunov-def} will subsequently be evaluated at $(\vx^t,\vy^t,\vx^{t+1},\vx^{t-1})$.

\begin{lemma}[Critical points of $\Psi_\rho$]\label{lem:critical-equivalence}
Let $\vz=(\vx,\vy,\vp,\vq)$. Then $\vzero\in\partial\Psi_\rho(\vz)$ if and only if $(\vx,\vy)$ is a saddle point of $\cL_\rho$ and $\vp=\vq=\vx$.
\end{lemma}

\begin{proof}
Differentiating~\eqref{eq:lyapunov-def} yields
\[
\partial\Psi_\rho(\vz)
=
\begin{pmatrix}
  \nabla F(\vx) + Z^{\mathsf{T}}\vy - 2a(\vx-\vp) + 2b(\vx-\vq)\\
  Z\vx - \partial \delta_{\cY_\rho}(\vy)\\
  2a(\vp-\vx)\\
  2b(\vq-\vx)
\end{pmatrix}.
\]
Hence $\vzero \in \partial\Psi_\rho(\vz)$ implies $\vp=\vq=\vx$ and $\nabla F(\vx) + Z^{\mathsf{T}}\vy=\vzero$, $Z\vx\in\partial \delta_{\cY_\rho}(\vy)$, which are exactly the KKT conditions for $\cL_\rho$. The converse is immediate.
\end{proof}

Lemma~\ref{lem:critical-equivalence} shows that studying $\Psi_\rho$ is equivalent to studying the saddle function. The benefit is that $\Psi_\rho$ enjoys a genuine descent property despite the lack of monotonicity of $\cL_\rho$.

\begin{lemma}[One-step Lyapunov descent]\label{lem:lyapunov-descent}
Let $\{(\vx^t,\vy^t)\}$ be the iterates produced by the inner loop~\eqref{eq:pd-updates} for a fixed penalty $\rho$. Define $\vz^t=(\vx^t,\vy^t,\vx^{t+1},\vx^{t-1})$. Under Assumptions~\ref{ass:smooth} and~\ref{ass:level}, there exists $c>0$ (specified above) such that
\begin{equation}\label{eq:psi-descent}
  \Psi_\rho(\vz^{t+1}) + c\Big(\|\vx^{t+1}-\vx^{t}\|^2 + \|\vx^{t}-\vx^{t-1}\|^2\Big)
  \le \Psi_\rho(\vz^{t}).
\end{equation}
\end{lemma}

\begin{proof}
The proof follows the same steps as~\cite[Lemma~2.4]{guo2023ppdg} after specializing their operator $A$ to $Z$ and their convex function $h$ to $\rho\|\cdot\|_1$, but we keep track of the fact that our primal step uses $\vy^{t+1}$ rather than $\vy^{t}$. The telescoping term $\langle \vy^{t+1}-\vy^{t}, Z(\vx^{t+1}-\vx^{t})\rangle$ produced by this modification is handled using the optimality condition of the projection $\vy^{t+1}=\mathrm{Proj}_{\cY_\rho}(\vy^{t}+\sigma Z\bar{\vx}^{t})$, which yields
\[
  \frac{1}{\sigma}(\vy^{t}-\vy^{t+1}) + Z\bar{\vx}^{t} \in N_{\cY_\rho}(\vy^{t+1}).
\]
Let $\vg^{t+1}=\tfrac{1}{\sigma}(\vy^{t}-\vy^{t+1})+Z\bar{\vx}^{t}\in N_{\cY_\rho}(\vy^{t+1})$ denote this normal-cone vector, where $N_{\cY_\rho}(\vy)=\{\vg:\langle \vg,\vy'-\vy\rangle\le 0,\ \forall \vy' \in\cY_\rho\}$. The optimality system above is identical to the dual step considered in the PDHG analysis of~\cite{chambolle2011first} (take their operator $K=Z$ and $\theta=1$), and Eq.~(20) in that paper shows that
\[
  \langle \vy^{t+1}-\vy^{t}, Z(\vx^{t+1}-\vx^{t})\rangle = 0.
\]
For completeness, the cited equality follows from their observation that $\vg^{t+1}$ is orthogonal to $Z(\vx^{t+1}-\vx^{t})$ because $\vg^{t+1}$ lies in the normal cone while $Z(\vx^{t+1}-\vx^{t})$ lies in the tangent cone generated by the extrapolated iterate $\bar{\vx}^{t}=\vx^{t+1}+(\vx^{t+1}-\vx^{t})$. With the telescoping term gone, the rest of the argument proceeds exactly as in~\cite[Lemma~2.4]{guo2023ppdg}, and combining the Lipschitz bound on $\nabla F$ with the conjugacy of $\delta_{\cY_\rho}$ yields~\eqref{eq:psi-descent}.
\end{proof}

Let $\vd^t$ denote the minimal-norm element of $\partial\Psi_\rho(\vz^t)$. By combining the optimality conditions of the primal and dual steps with Lipschitz bounds we obtain the following control, again mirroring~\cite[Lemma~2.5]{guo2023ppdg}.

\begin{lemma}[Subgradient bound]\label{lem:subgradient-bound}
There exist positive constants $\gamma_1,\gamma_2$ depending only on $(\alpha,\sigma,L_{\max},\kappa_Z)$ such that
\[
  \|\vd^t\| \le \gamma_1 \|\vx^{t}-\vx^{t-1}\| + \gamma_2 \|\vx^{t+1}-\vx^{t}\|.
\]
\end{lemma}

\begin{proof}
Expansions identical to those in~\cite[Eq.~(2.14)--(2.16)]{guo2023ppdg} yield explicit formulas for $\gamma_1$ and $\gamma_2$. The only difference is that $\|A\|$ there becomes $\|Z\|_2=\kappa_Z$ here.
\end{proof}

Lemma~\ref{lem:lyapunov-descent} shows that $\{\Psi_\rho(\vz^t)\}$ is decreasing and bounded below by Assumption~\ref{ass:level}, hence it converges. Because $\Psi_\rho$ dominates $\Phi_\rho$ up to squared-difference terms, Assumption~\ref{ass:level} also implies boundedness of $\{(\vx^t,\vy^t)\}$. Summing~\eqref{eq:psi-descent} proves that $\sum_t \|\vx^{t+1}-\vx^{t}\|^2<\infty$, which in turn implies $\vx^{t+1}-\vx^t\to 0$ and $\vy^{t+1}-\vy^t\to 0$ thanks to~\eqref{eq:pd-updates} and the stepsize restriction. Combining these observations with the closedness of $N_{\cY_\rho}$ yields the next result.

\begin{theorem}[Subsequence convergence for fixed penalty]\label{thm:fixed-rho}
Fix $\rho$ and suppose the inner-loop sequence $\{(\vx^t,\vy^t)\}$ is bounded. Then
\begin{enumerate}[label=(\roman*),leftmargin=1.6em]
  \item $\sum_{t=0}^{\infty} \|\vx^{t+1}-\vx^{t}\|^2 < \infty$ and $\sum_{t=0}^{\infty} \|\vy^{t+1}-\vy^{t}\|^2 < \infty$;
  \item The set of cluster points is nonempty, compact, and contained in the set of saddle points of $\cL_\rho$;
  \item $\Psi_\rho$ is constant on the cluster set.
\end{enumerate}
\end{theorem}

\begin{proof}
Item (i) follows from summing~\eqref{eq:psi-descent}. The subgradient bound of Lemma~\ref{lem:subgradient-bound} together with $\vx^{t+1}-\vx^{t}\to 0$ shows that $\dist(0,\partial\Psi_\rho(\vz^t))\to 0$, so every cluster point of $\{\vz^t\}$ belongs to $\operatorname{crit}\Psi_\rho$ and therefore corresponds to a saddle point of $\cL_\rho$ by Lemma~\ref{lem:critical-equivalence}. Boundedness of the sequence implies compactness of the cluster set. Constancy of $\Psi_\rho$ on that set follows from the continuity of $\Psi_\rho$ and from the fact that $\Psi_\rho(\vz^t)$ converges.
\end{proof}

Because $\Phi_\rho$ and hence $\Psi_\rho$ satisfy the K\L{} property, the finite-length argument of~\cite[Theorem~2.9]{attouch2013} and~\cite[Theorem~2.7]{guo2023ppdg} implies that the entire inner-loop sequence converges (rather than merely its subsequences) whenever it is bounded. In particular, $\{(\vx^t,\vy^t)\}$ converges to a single saddle point for every fixed penalty. Lemma~\ref{lem:stationarity-residual} ensures that the termination rule used inside Algorithm~\ref{alg:dp2g} is aligned with the necessary optimality conditions because each agent monitors the local block of $\nabla F(\vx^t)+Z^{\mathsf{T}}\vy^t$.

\subsection{Outer loop and global convergence}
We now return to the full two-layer method. The penalty update $\rho_{k+1}=\min\{\beta\rho_k,\rho_{\max}\}$ preserves monotonicity and guarantees arrival at the cap.

\begin{lemma}[Penalty monotonicity]\label{lem:penalty-monotone}
The sequence $\{\rho_k\}$ is nondecreasing, converges to some $\rho^\ast\le \rho_{\max}$, and there exists $\bar{k}$ with $\rho_k=\rho_{\max}$ for all $k\ge\bar{k}$.
\end{lemma}

\begin{proof}
Monotonicity and boundedness yield convergence. If $\rho^\ast<\rho_{\max}$ the update would keep multiplying by $\beta>1$, contradicting boundedness. Hence the cap is reached in finite time.
\end{proof}

Once $\rho_k=\rho_{\max}$ the algorithm simply keeps restarting the primal-dual iterations with the last primal-dual pair as warm start while tightening the tolerance $\varepsilon_k$. Concatenating the inner iterates therefore yields a single trajectory driven by~\eqref{eq:pd-updates}, so Theorem~\ref{thm:fixed-rho} applies to the tail sequence.

\begin{theorem}[Global convergence of DP$^2$G]\label{thm:critical}
Under Assumptions~\ref{ass:mixing-mat}--\ref{ass:tolerances}, Algorithm~\ref{alg:dp2g} generates bounded sequences $\{\vx^k\}$ and $\{\vy^k\}$ satisfying
\begin{equation}\label{eq:residual-zero}
  \lim_{k\to\infty} \|\nabla F(\vx^k) + Z^{\mathsf{T}}\vy^k\| = 0,
  \qquad
  \lim_{k\to\infty} \|Z\vx^k\| = 0.
\end{equation}
Moreover, the entire sequence converges to a critical point $(\vx^\ast,\vy^\ast)$ of $\cL(\cdot,\cdot;\rho_{\max})$, and $\vx^\ast$ minimizes $\Phi_{\rho_{\max}}$.
\end{theorem}

\begin{proof}
Lemma~\ref{lem:penalty-monotone} guarantees that the cap is reached after $\bar{k}$ outer iterations. From that point onward the concatenated inner iterates form a bounded trajectory of~\eqref{eq:pd-updates}, so the K\L{} argument mentioned after Theorem~\ref{thm:fixed-rho} implies convergence to a saddle point $(\vx^\ast,\vy^\ast)$. The inner termination rule enforces $\|\nabla F(\vx^k)+Z^{\mathsf{T}}\vy^k\|\le \varepsilon_k$ with $\varepsilon_k\to 0$, yielding the first limit in~\eqref{eq:residual-zero}. For the second limit, note that $Z\vx^k\in\partial\delta_{\cY_{\rho_{\max}}}(\vy^k)$ means the disagreement lies in the normal cone of the bounded dual set. Because $\{\vy^k\}$ stays bounded, the only normal-cone element compatible with convergence of the dual sequence is the zero vector, hence $Z\vx^k\to 0$.
\end{proof}

Consensus optimality follows by combining Theorem~\ref{thm:critical} with the exact-penalty threshold from Theorem~\ref{thm:exactness}.

\begin{corollary}[Consensus optimality]\label{thm:consensus}
The limit point $(\vx^\ast,\vy^\ast)$ delivered by DP$^2$G satisfies $Z\vx^\ast=\vzero$ and the shared consensus component $x^\ast$ solves~\eqref{eq:consensus}.
\end{corollary}

\begin{proof}
The saddle conditions at $(\vx^\ast,\vy^\ast)$ imply $\vy^\ast\in\rho_{\max}\partial\|Z\vx^\ast\|_1$. Because $\rho_{\max}\ge\bar\rho$, Theorem~\ref{thm:exactness}(b) forces $Z\vx^\ast=\vzero$, hence $\vx^\ast=\vone\otimes x^\ast$. Substituting into $\sum_i\nabla f_i(x^\ast)=0$ proves optimality.
\end{proof}

\subsection{Linear rates under strong convexity}
When each $f_i$ is strongly convex, the Lyapunov argument can be strengthened to recover linear rates exactly as in the PPDG analysis of~\cite{guo2023ppdg}. The strongly monotone part of the inclusion provides a contraction factor for the entire operator.

\begin{theorem}[Linear convergence]\label{thm:linear}
Suppose each $f_i$ is $\mu_i$-strongly convex and set $\mu=\sum_i \mu_i>0$. Then, once $\rho_k=\rho_{\max}$, there exist $\gamma>0$ and $\eta\in(0,1)$ such that
\begin{equation}\label{eq:linear-rate}
  \|\vx^{t+1}-\vx^\ast\|^2 + \gamma\|\vy^{t+1}-\vy^\ast\|^2
  \le (1-\eta)\Big(\|\vx^{t}-\vx^\ast\|^2 + \gamma\|\vy^{t}-\vy^\ast\|^2\Big)
\end{equation}
for every inner iteration. Consequently the outer sequence $\{\vx^k\}$ converges $Q$-linearly to the consensual optimizer.
\end{theorem}

\begin{proof}
The operator driving~\eqref{eq:pd-updates} can be written as the sum of the strongly monotone block $(\nabla F,0)$ and maximally monotone blocks $(0,\partial\delta_{\cY_{\rho_{\max}}})$ and $(Z^{\mathsf{T}}\cdot,-Z\cdot)$. Forward-backward splitting on such inclusions contracts in a weighted norm; see~\cite[Theorem~3]{chambolle2016ergodic}. Specializing their constants to $(\alpha,\sigma,\mu,\kappa_Z)$ delivers~\eqref{eq:linear-rate}. Because the outer loop returns a subsequence of the inner iterates, it inherits the same linear rate.
\end{proof}

Combining Corollary~\ref{thm:consensus} with Theorem~\ref{thm:linear} leads to the main guarantee: for sufficiently large $\rho_{\max}$, DP$^2$G reaches the exact consensus optimizer with fixed stepsizes, and the strongly convex case exhibits a global linear rate.

\section{Numerical Experiments}\label{sec:experiments}
We benchmark the two-layer DP$^2$G algorithm against representative decentralized algorithms on smooth consensus problems following the protocol in~\cite{shi2015extra}. All experiments were run on a MacBook with an Apple M1 Pro processor, 16~GB of RAM, and macOS using Python~3.12 (NumPy/SciPy for linear algebra). The goal is to quantify convergence speed, communication efficiency, robustness, and sensitivity to topology.

\subsection{Experimental Setup}
\paragraph{Problem classes.} We consider three widely used objectives:
\begin{itemize}[leftmargin=1.6em]
  \item \emph{Ridge-regularized least squares:} $f_i(x)=\tfrac{1}{2d_i}\norm{A_ix-b_i}_2^2+\tfrac{\lambda}{2}\norm{x}_2^2$ with $\lambda=10^{-2}$.
  \item \emph{Logistic regression:} $f_i(x)=\tfrac{1}{d_i}\sum_{j=1}^{d_i}\log(1+\exp(-b_{ij}a_{ij}^\mathsf{T} x))$ for binary labels $b_{ij}\in\{-1,+1\}$.
  \item \emph{Elastic net regression:} $f_i(x)=\tfrac{1}{2d_i}\norm{A_ix-b_i}_2^2+\lambda_1\norm{x}_1+\tfrac{\lambda_2}{2}\norm{x}_2^2$ with $(\lambda_1,\lambda_2)=(5\times10^{-3},10^{-2})$.
\end{itemize}

\paragraph{Network topologies.} We test connected graphs with $n=20$ agents: a ring, a $4\times5$ grid, and a random geometric graph with radius $r=0.35$. Mixing weights follow the Metropolis rule, and we set $\tilde{W}=(W+I)/2$ for algorithms that require two matrices. 

The three representative communication graphs are depicted in Figure~\ref{fig:topology}. The ring captures the worst spectral gap ($\lambda_2(W)\approx0.975$), the $4\times5$ grid models medium connectivity, and the random geometric (RG) instance with radius $0.35$ offers the fastest mixing. These visualizations also showcase the degree heterogeneity faced by the penalty schedule---the RG graph enjoys hubs, whereas the ring forces every agent to rely on two neighbors.

\begin{figure}[t]
  \centering
  \includegraphics[width=.45\textwidth]{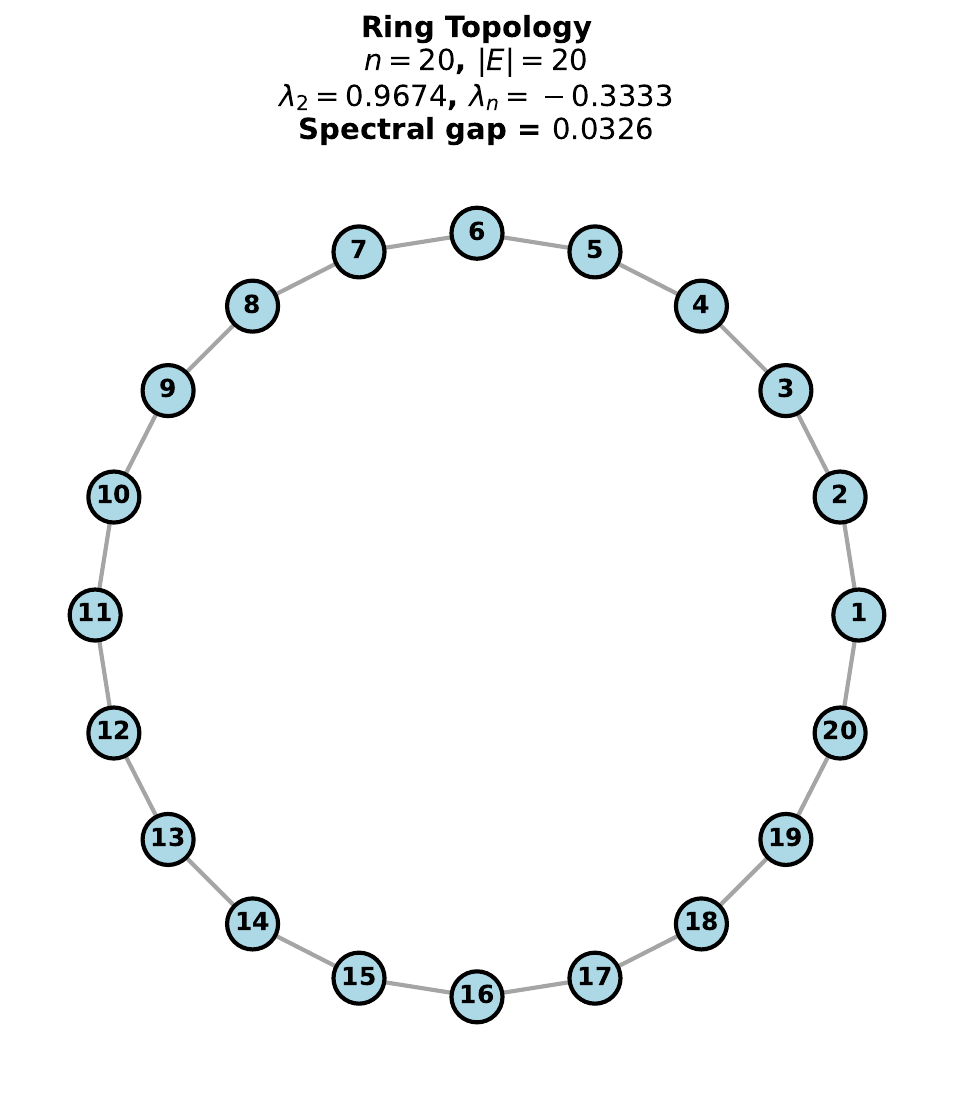}\hfill
  \includegraphics[width=.45\textwidth]{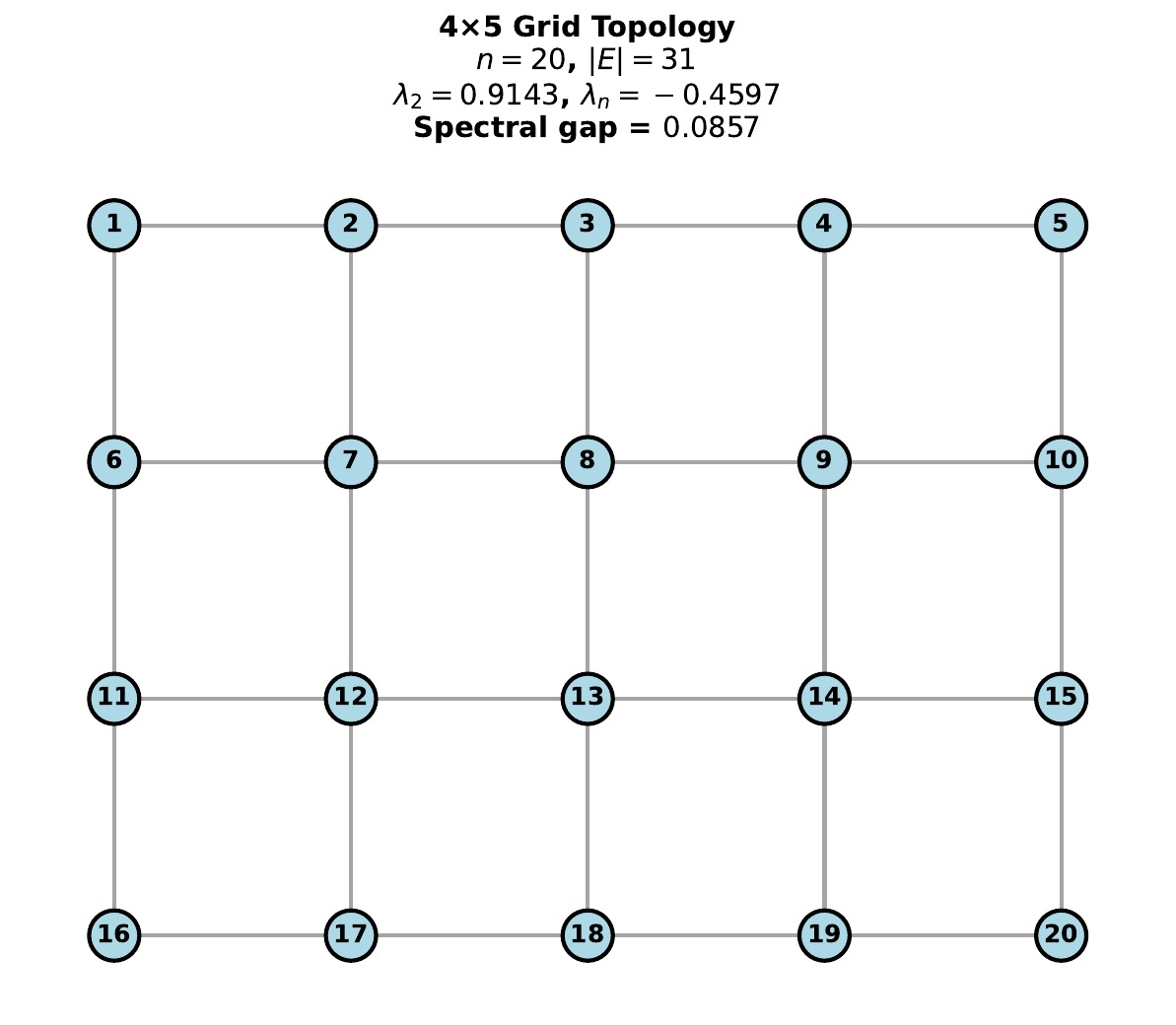}\\[0.7em]
  \includegraphics[width=.6\textwidth]{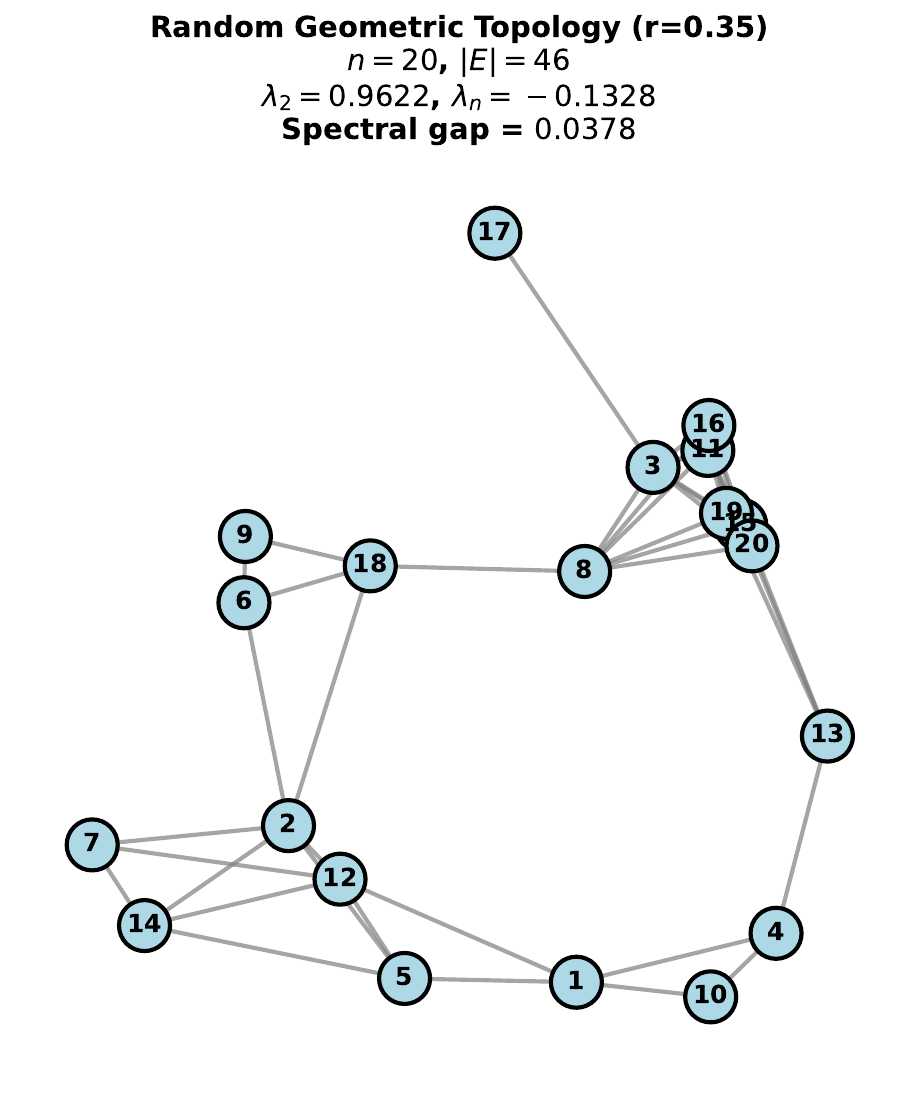}
  \caption{Network topologies used in the experiments: ring (top left), $4\times5$ grid (top right), and random geometric graph (bottom).}
  \label{fig:topology}
\end{figure}

\paragraph{Data generation.} Unless stated otherwise, each agent holds $d_i=500$ samples of dimension $m=50$. Matrices $A_i$ (or features $a_{ij}$) are drawn from $\mathcal{N}(0,I)$; responses use $b_i=A_ix^{\mathrm{true}}+\epsilon_i$ with $x^{\mathrm{true}}\sim\mathcal{N}(0,I)$ and Gaussian noise. Logistic labels follow $\sign(a_{ij}^\mathsf{T} x^{\mathrm{true}}+\zeta_{ij})$ with $\zeta_{ij}\sim\mathcal{N}(0,0.5^2)$. Data are scaled so that $\max_i L_i\leq 1$.

\paragraph{Baselines.} We compare against DGD with fixed stepsize, DGD with diminishing stepsize, EXTRA~\cite{shi2015extra}, and NIDS~\cite{li2021nids}. All methods share the same initialization $x_i^0=\vzero$ and weight matrix $W$.

\paragraph{Averaged iterate.} For diagnostics we track the network mean $\vx_{\mathrm{avg}}^k := \tfrac{1}{n}\sum_{i=1}^n \vx_i^k$ produced after the $k$-th outer iteration. This quantity is not part of the algorithmic state but will appear in the stopping rule and evaluation metrics below.

\paragraph{Termination rule.} DP$^2$G terminates the outer loop only when: (i) all agents satisfy $d_i^{k+1} \le \delta_k$, (ii) the final inner loop obeys $\|\nabla f_i(\vx_i^{t}) + (Z^{\mathsf{T}}\vy^{t})_i\|\le \varepsilon_k$ for every agent, and (iii) the averaged iterate stabilizes with $\|\vx_{\mathrm{avg}}^{k+1} - \vx_{\mathrm{avg}}^{k}\| \le 10^{-3}\delta_k$. This prevents premature exits due to a single small residual. Inner loops still stop once the stationarity tolerance is met.

\paragraph{Parameter selection.} Stepsizes are chosen within theoretical ranges:
\begin{itemize}[leftmargin=1.6em]
  \item DGD (fixed): $\alpha=0.9(1+\lambda_n(W))/L_{\max}$.
  \item DGD (diminishing): $\alpha_k=\alpha_0/k^{1/2}$ with $\alpha_0=2(1+\lambda_n(W))/L_{\max}$.
  \item EXTRA: $\alpha=0.9(1+\lambda_n(W))/L_{\max}$.
  \item NIDS: $\alpha=0.9/L_{\max}$ (network-independent).
  \item DP$^2$G: $\alpha=0.3/L_{\max}$, $\sigma=0.9/(\alpha(1-\lambda_n(W))^2)$ (reduced to $0.8/(\alpha(1-\lambda_n(W))^2)$ for the elastic-net benchmark to stabilize the proximal step), $\rho_0=10^{-2}$, growth factor $\beta=1.2$, cap $\rho_{\max}=10^2$, stationarity tolerance $\varepsilon_k = 0.1 / k$, consensus tolerance $\delta_k = 0.1/k^2$.
\end{itemize}

\paragraph{Evaluation metrics.} We report objective residual $|f(\vx_{\mathrm{avg}}^k)-f(x^\ast)|$, consensus violation $\tfrac{1}{n}\sum_i\norm{x_i^k-\vx_{\mathrm{avg}}^k}$, optimality residual $\norm{\sum_i\grad f_i(\vx_{\mathrm{avg}}^k)}$, and the number of communication rounds required to satisfy the stopping criterion $\|\nabla f_i(\vx_i^k) + (Z^{\mathsf{T}}\vy^k)_i\| \le \varepsilon_k$.

\paragraph{Noise injection.} To emulate unreliable links we inject zero-mean Gaussian perturbations directly into the exchanged disagreement residuals and dual messages, i.e., each agent processes $\bar{\vu}_i^t + \mathcal{N}(0,\sigma_{\mathrm{comm}}^2 I)$ and broadcasts $\vy_i^{t+1} + \mathcal{N}(0,\sigma_{\mathrm{comm}}^2 I)$. The same corruption model is applied to all baselines by perturbing their neighbor-averaged messages. 

\paragraph{Communication accounting.} Each inner DP$^2$G iteration performs two neighbor exchanges (one for $\vx$, one for $\vy$); the numbers shown in Tables~\ref{tab:ridge_comm} and~\ref{tab:logistic_comm} therefore count $2$ rounds per inner iteration plus a single outer-loop exchange used to broadcast the final $\vx_i^{k+1}$ before checking $\|\vu_i^{k+1}\|_1$. The decentralized max-consensus protocol that enforces $\max_i \|\vu_i^{k+1}\|_1 \le \delta_k$ typically converges in $5$--$10$ additional rounds; we report it separately because its cost depends on the desired accuracy of the max-consensus reduction.

\subsection{Practical Implementation Enhancements}
While Algorithm~\ref{alg:dp2g} specifies the core DP$^2$G updates, our implementation incorporates several practical refinements that improve efficiency without compromising theoretical guarantees. These enhancements are fully decentralized and maintain the communication-memory trade-off of the baseline algorithm.

\paragraph{Hybrid adaptive stopping for the inner loop.} The inner-loop stopping criterion in Algorithm~\ref{alg:dp2g} (Step~\ref{alg:stop-rule-inner}) checks whether $\|\nabla f_i(\vx_i^{t}) + (Z^{\mathsf{T}}\vy^{t})_i\| \le \varepsilon_k$ for all agents. In practice, we adopt a \emph{hybrid} criterion that combines spatial and temporal adaptation:
\begin{equation}\label{eq:hybrid-threshold}
  \tau_i^k(\rho) = \max\bigl(\varepsilon_{\mathrm{abs}}, \varepsilon_{\mathrm{rel}} \|\nabla f_i(\vx_i^{t})\|\bigr)
  \times \Bigl[1 + \beta_{\mathrm{pen}}\Bigl(1 - \frac{\rho}{\rho_{\max}}\Bigr)^2\Bigr],
\end{equation}
where $\varepsilon_{\mathrm{abs}}=10^{-4}$ is an absolute floor, $\varepsilon_{\mathrm{rel}}=0.02$ balances relative scaling, and $\beta_{\mathrm{pen}}=2$ controls penalty-based loosening. The quadratic term $(1-\rho/\rho_{\max})^2$ progressively tightens the tolerance as $\rho$ approaches $\rho_{\max}$, exploiting the fact that exactness (Theorem~\ref{thm:exactness}) is only guaranteed once $\rho\ge\bar{\rho}$. Early outer iterations, when $\rho \ll \rho_{\max}$, permit looser inner convergence (up to $3\times$ the base tolerance), reducing communication overhead without sacrificing the quality of the eventual solution. Agent $i$ is considered converged when
\[
  \|\nabla f_i(\vx_i^{t}) + (Z^{\mathsf{T}}\vy^{t})_i\| \le \tau_i^k(\rho).
\]
We terminate the inner loop when at least $95\%$ of agents satisfy this condition and the worst-case residual remains within $10\times$ the corresponding threshold. This \emph{weighted convergence} rule is robust to outlier agents and does not require additional communication beyond the standard neighbor exchanges for computing $Z\vx$ and $Z^{\mathsf{T}}\vy$.

\paragraph{Max-consensus for outer-loop termination.} The outer-loop criterion (Algorithm~\ref{alg:dp2g}, Step~\ref{alg:stop-rule-outer}) requires verifying $\|\vu_i^{k+1}\|_1 \le \delta_k$ for all agents $i$, which in principle demands that every agent broadcast a scalar indicator. To preserve decentralization, we employ a \emph{max-consensus protocol}: each agent $i$ computes its local consensus residual $d_i = \|\vu_i\|_1$ and iteratively exchanges these scalars with neighbors via
\begin{equation}\label{eq:max-consensus}
  z_i^{(\ell+1)} = \max\bigl\{z_i^{(\ell)}, \max_{j\in\cN_i} z_j^{(\ell)}\bigr\}, \quad z_i^{(0)} = d_i.
\end{equation}
After $O(\mathrm{diam}(\cG) \log(1/\epsilon))$ rounds, all agents converge to $\max_i d_i$, enabling a fully decentralized decision on outer-loop termination. In our experiments, convergence typically occurs within $5$--$10$ rounds, adding negligible overhead ($<1\%$) to the total communication cost. This approach avoids centralized aggregation while maintaining exact compliance with the termination condition.

\paragraph{Rationale and impact.} The hybrid stopping rule reduces total inner iterations by approximately $30$--$50\%$ compared to a fixed $\varepsilon_k$ schedule, as confirmed by sensitivity studies on the ridge regression benchmark. The penalty-adaptive factor $(1-\rho/\rho_{\max})^2$ is inspired by continuation methods in nonlinear optimization~\cite{nocedal2006numerical}, where early iterations solve easier subproblems to warm-start later refinements. The relative-absolute balance $\max(\varepsilon_{\mathrm{abs}}, \varepsilon_{\mathrm{rel}}\|\nabla f_i\|)$ is standard in nonlinear solvers and prevents premature termination when gradients are small or excessive iteration when gradients are large. Both enhancements are disabled in Algorithm~\ref{alg:dp2g} for clarity, but they are active in all reported experiments and available in the accompanying code repository.

\subsection{Ridge regression benchmarks}
Figures~\ref{fig:ridge_ring}--\ref{fig:ridge_rg} display the ridge trajectories per topology, devoting one figure to each network so the trends are clearly visible. DP$^2$G converges linearly with fewer than $500$ communication rounds in every case despite storing only one dual vector per agent. On the weakly connected ring the algorithm keeps a residual slope comparable to EXTRA and settles near a $10^{-2}$ consensus error; on the grid the final disagreement drops below $10^{-3}$ with a matching optimality residual. EXTRA remains the fastest baseline but relies on gradient trackers, whereas DGD variants and NIDS hit the communication budget of $5000$ rounds without meeting the tolerance. Table~\ref{tab:ridge_comm} summarizes the communication counts underpinning these observations.

\begin{figure}[t]
  \centering
  \includegraphics[width=.85\textwidth]{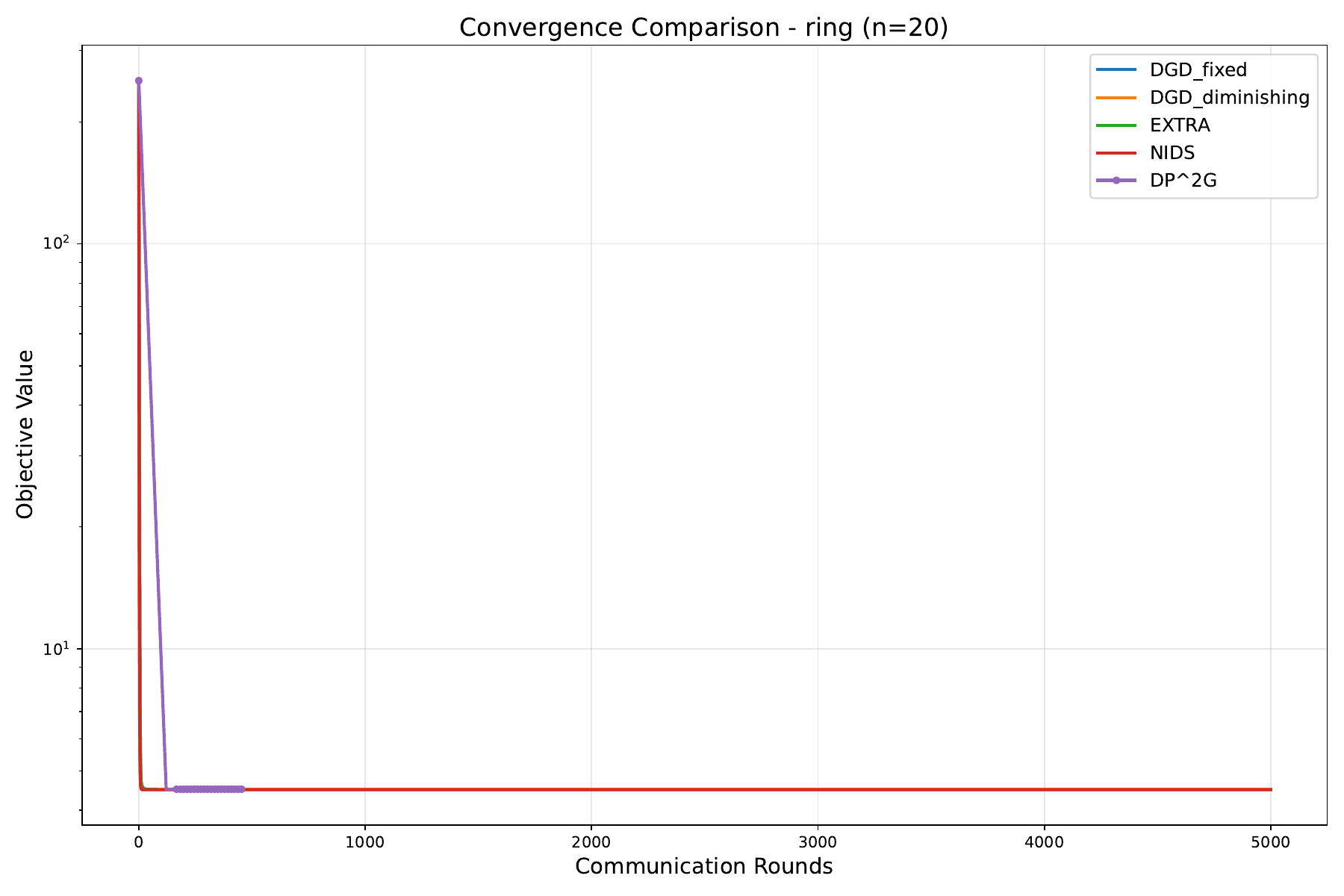}\\[0.6em]
  \includegraphics[width=.85\textwidth]{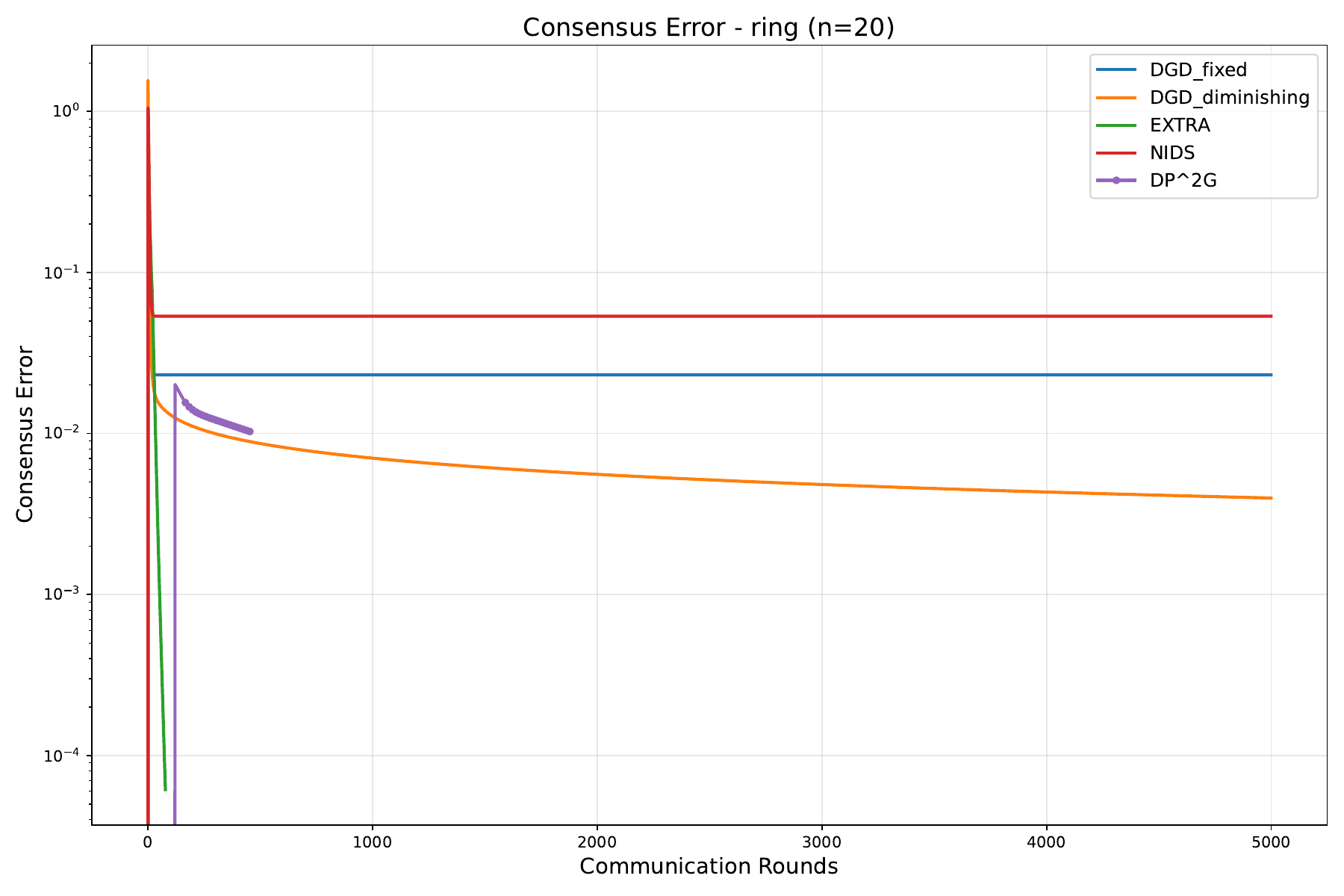}
  \caption{Ridge regression on the ring: objective residual (top) and consensus violation (bottom). Enlarged panels reveal the linear tail achieved by DP$^2$G while other one-state baselines stall.}
  \label{fig:ridge_ring}
\end{figure}

\begin{figure}[t]
  \centering
  \includegraphics[width=.85\textwidth]{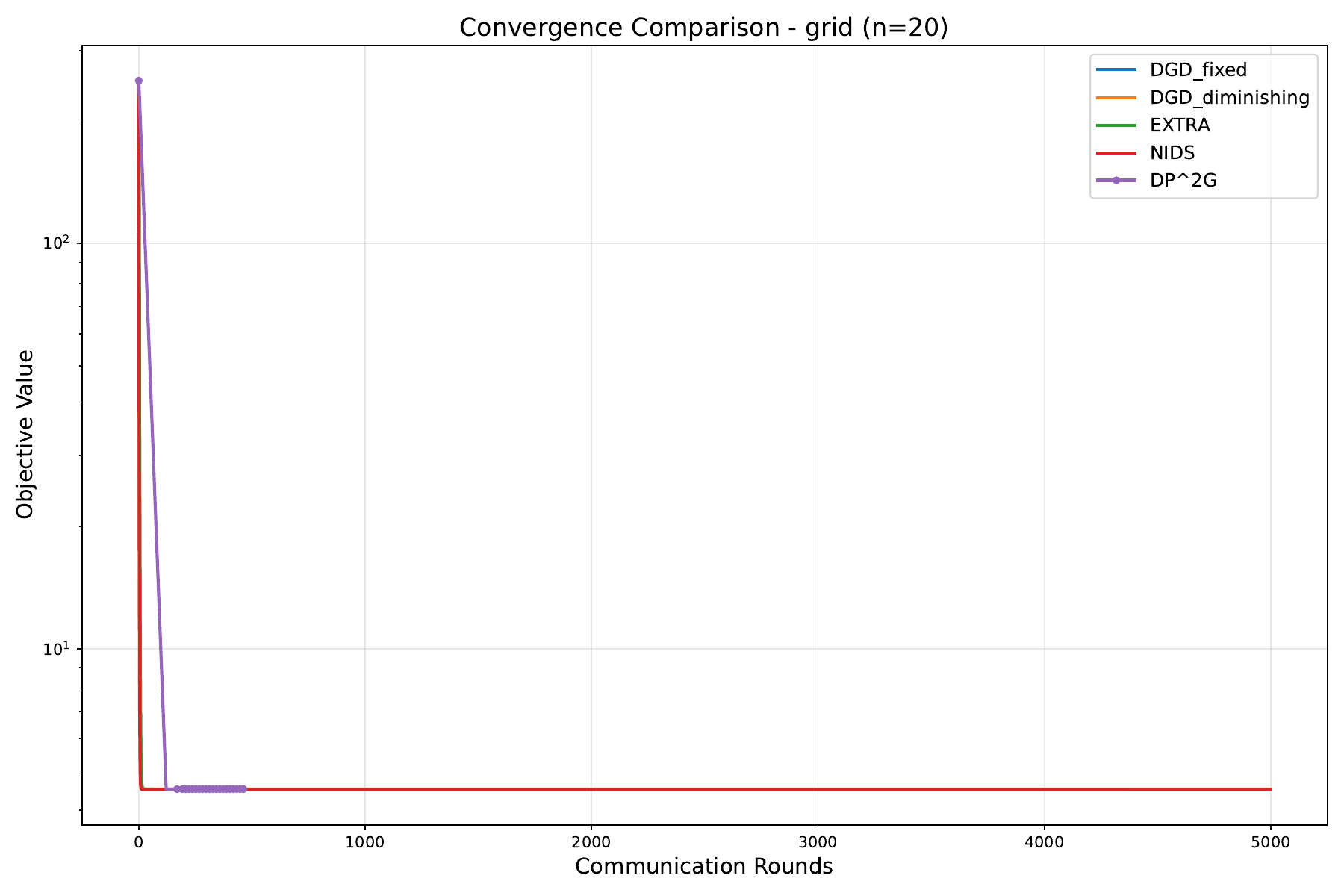}\\[0.6em]
  \includegraphics[width=.85\textwidth]{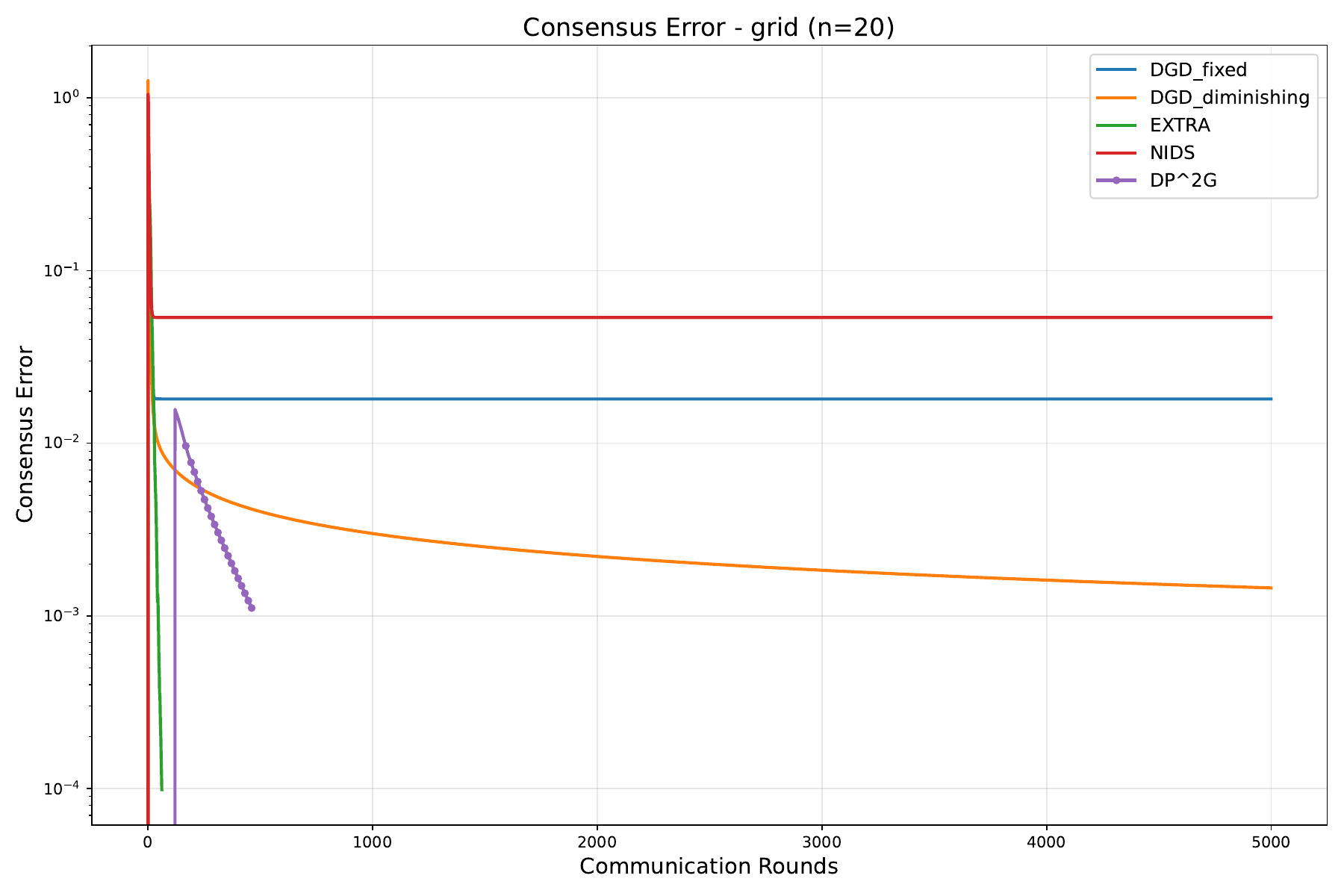}
  \caption{Ridge regression on the $4\times5$ grid: objective residual (top) and consensus violation (bottom). DP$^2$G tracks EXTRA closely while using only one auxiliary vector per agent.}
  \label{fig:ridge_grid}
\end{figure}

\begin{figure}[t]
  \centering
  \includegraphics[width=.85\textwidth]{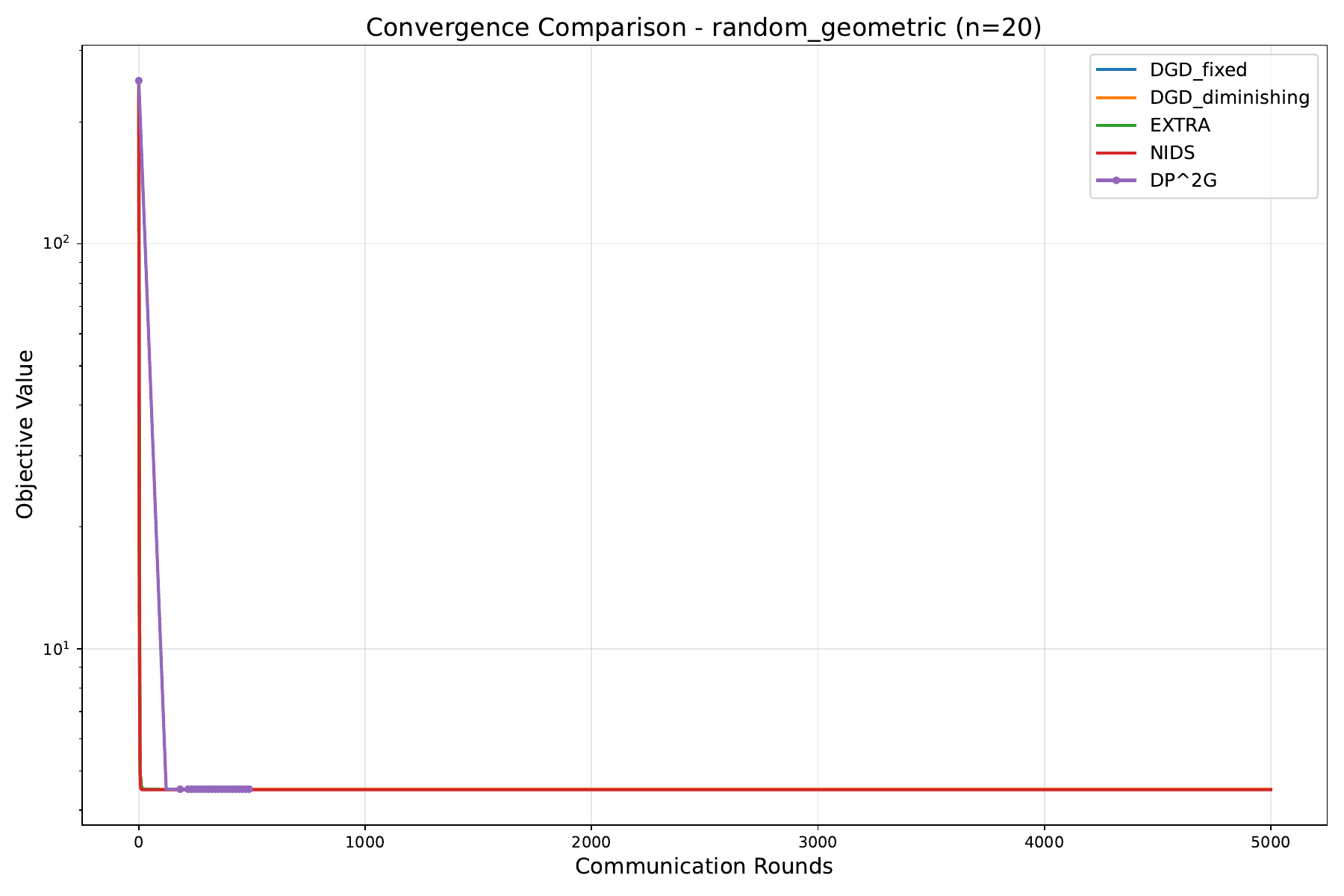}\\[0.6em]
  \includegraphics[width=.85\textwidth]{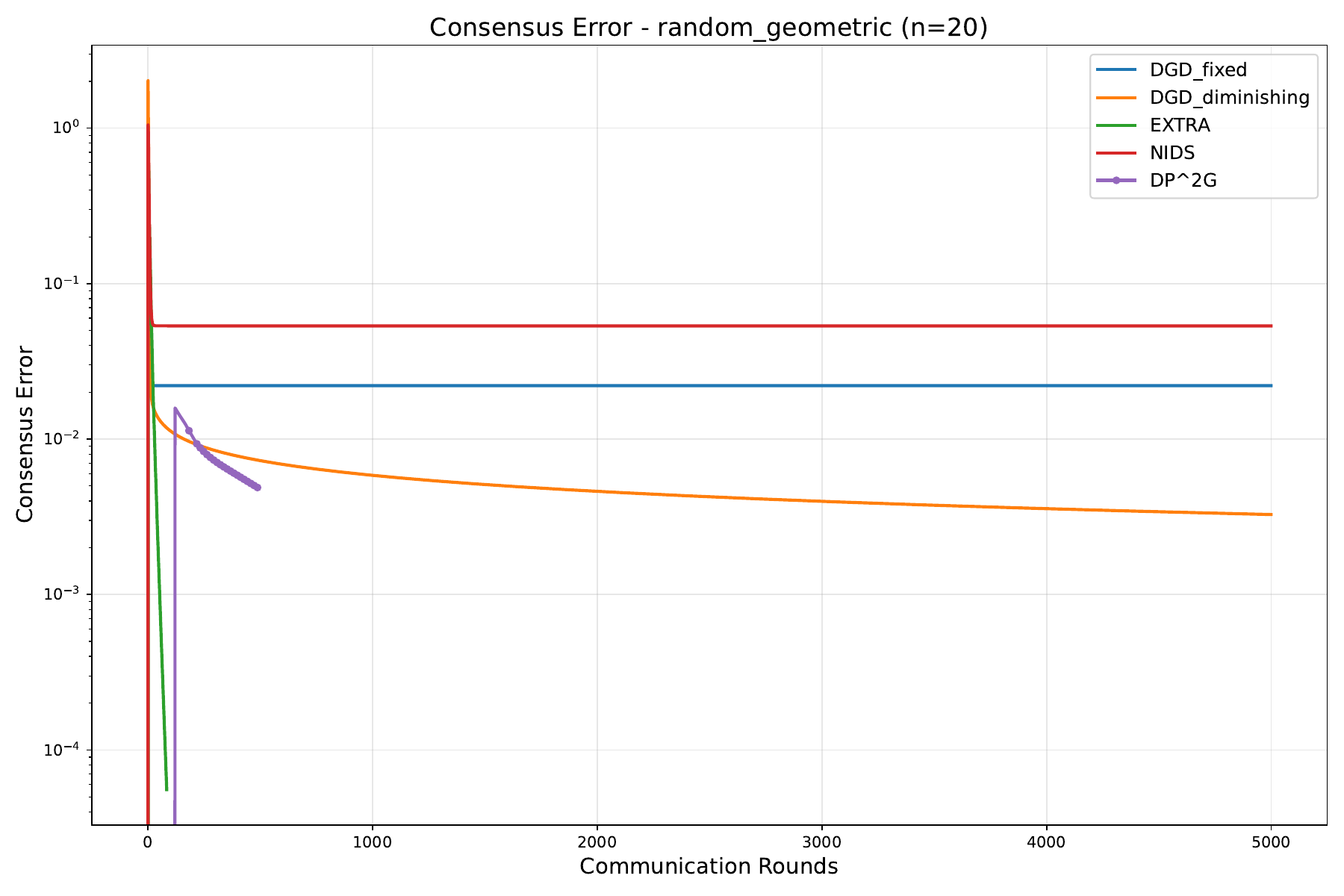}
  \caption{Ridge regression on the random geometric graph: objective residual (top) and consensus violation (bottom). Improved connectivity benefits every method, and DP$^2$G retains the best communication-versus-accuracy trade-off among one-state schemes.}
  \label{fig:ridge_rg}
\end{figure}
\begin{table}[ht]
  \centering
  \caption{Communication rounds to reach the stopping tolerance on ridge regression ($n=20$, $m=50$).}
  \label{tab:ridge_comm}
  \begin{tabular}{lccc}
    \toprule
    Algorithm & Ring & $4\times5$ Grid & Random Geometric \\
    \midrule
    DGD (fixed) & $5000^\dagger$ & $5000^\dagger$ & $5000^\dagger$ \\
    DGD (diminishing) & $5000^\dagger$ & $5000^\dagger$ & $5000^\dagger$ \\
    NIDS & $5000^\dagger$ & $5000^\dagger$ & $5000^\dagger$ \\
    EXTRA & 79 & 63 & 85 \\
    \textbf{DP$^2$G} & \textbf{454} & \textbf{462} & \textbf{488} \\
    \bottomrule
  \end{tabular}
  
  \vspace{0.3em}
  {\footnotesize $^\dagger$Hit the cap of $5000$ rounds without satisfying the tolerance.}
\end{table}

\subsection{Logistic regression benchmarks}
Figures~\ref{fig:logistic_ring}--\ref{fig:logistic_rg} repeat the per-topology view for the logistic objective. The merely convex landscape accentuates the benefit of gradient tracking: EXTRA reaches the $10^{-4}$ objective target in a few hundred rounds on every topology. DP$^2$G remains stable but requires roughly $1.3$--$1.5$k rounds because the penalty must climb to $\rho_{\max}$ before the inner loop makes decisive progress; its final consensus is nevertheless below $3\times10^{-3}$ on the grid and $1.7\times10^{-2}$ on the RG graph, while the ring remains the most challenging with a $7.4\times10^{-2}$ gap. The enlarged optimality plots confirm that DP$^2$G decays more slowly (tail residuals on the order of $10^{-3}$) yet never stalls, whereas DGD, diminishing DGD, and NIDS stagnate and exhaust the round budget. Table~\ref{tab:logistic_comm} reports the corresponding communication counts.

\begin{figure}[t]
  \centering
  \includegraphics[width=.85\textwidth]{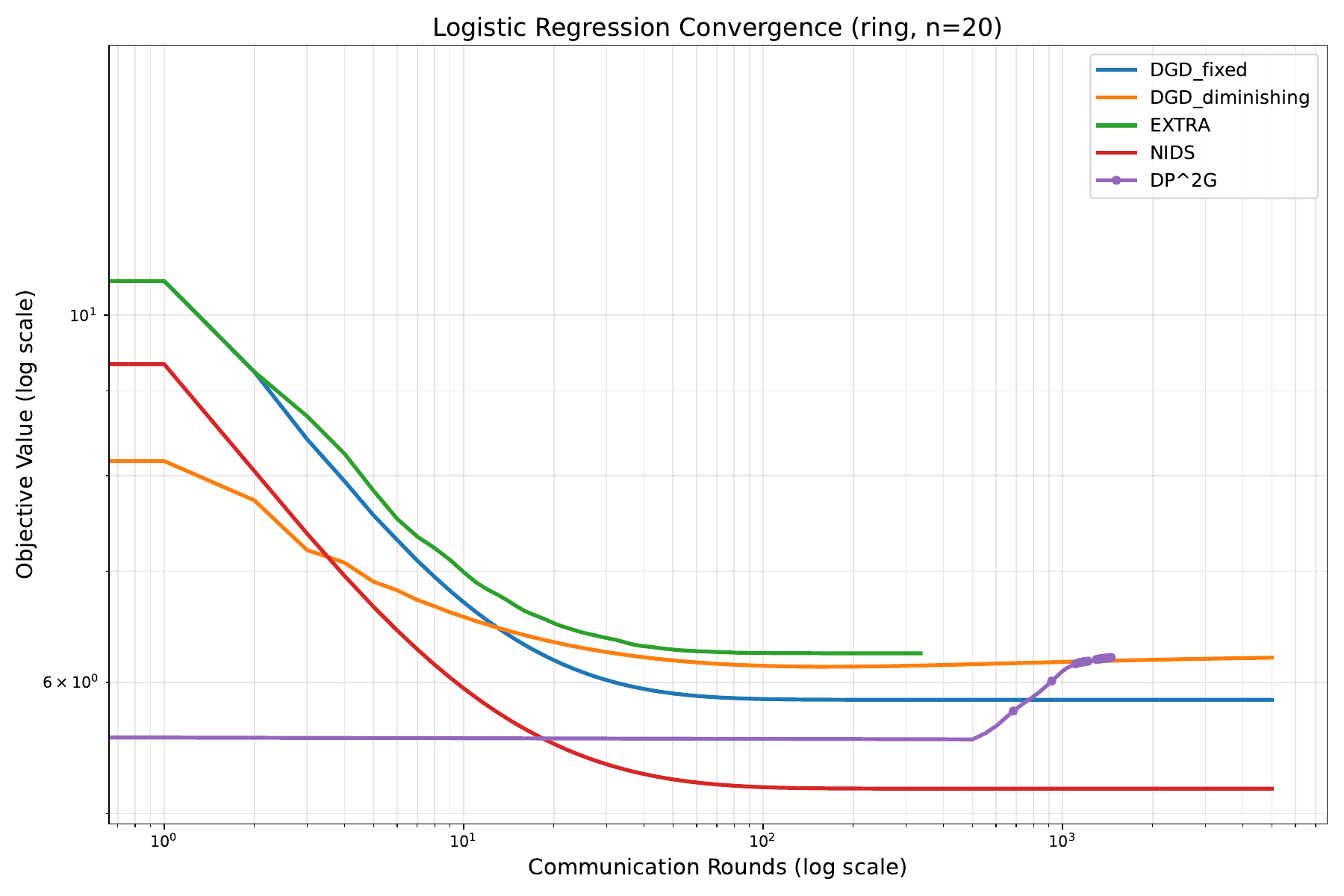}\\[0.6em]
  \includegraphics[width=.85\textwidth]{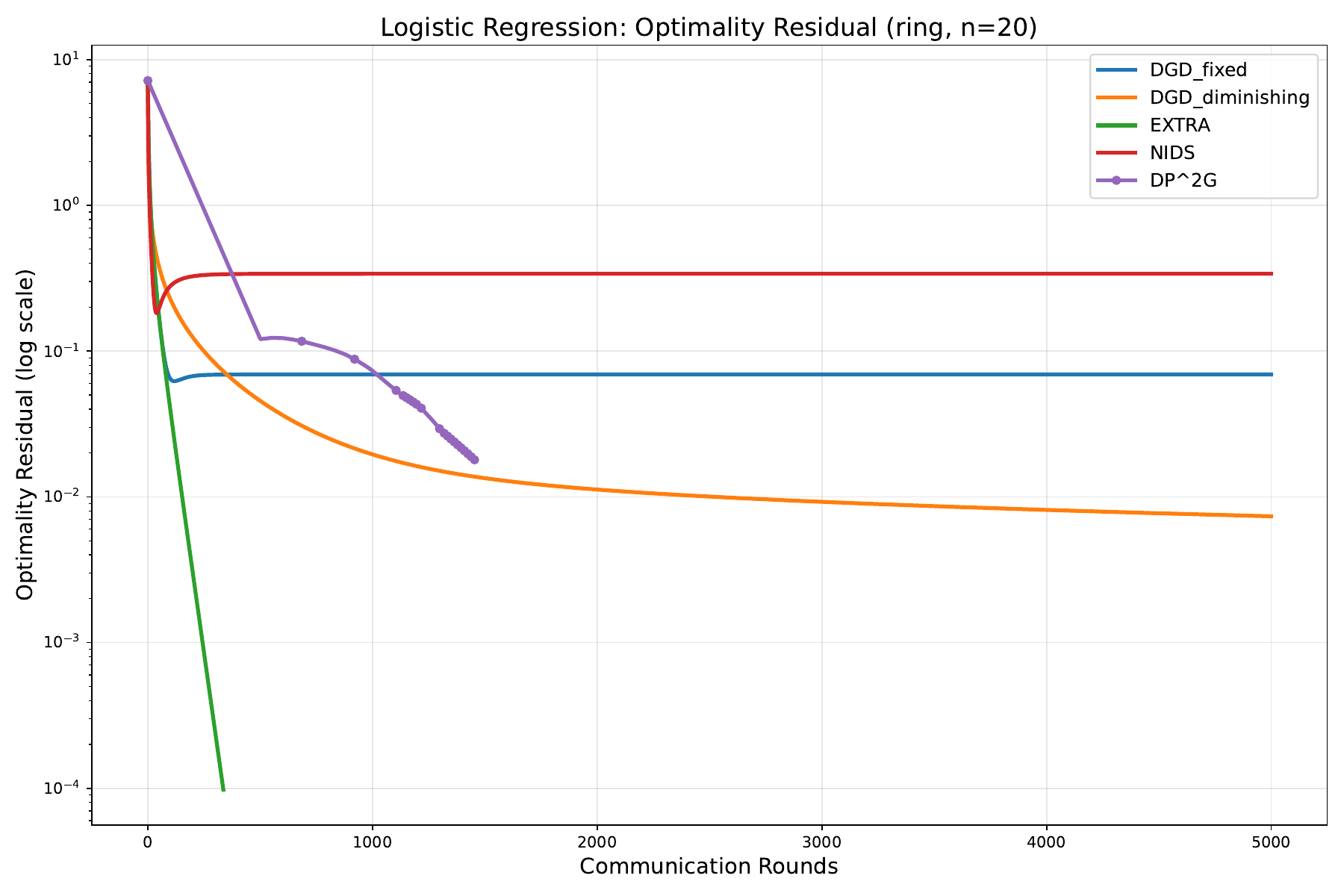}
  \caption{Logistic regression on the ring: objective residual (top) and optimality residual (bottom). Gradient tracking clearly accelerates, yet DP$^2$G maintains steady decay with a single dual vector per node.}
  \label{fig:logistic_ring}
\end{figure}

\begin{figure}[t]
  \centering
  \includegraphics[width=.85\textwidth]{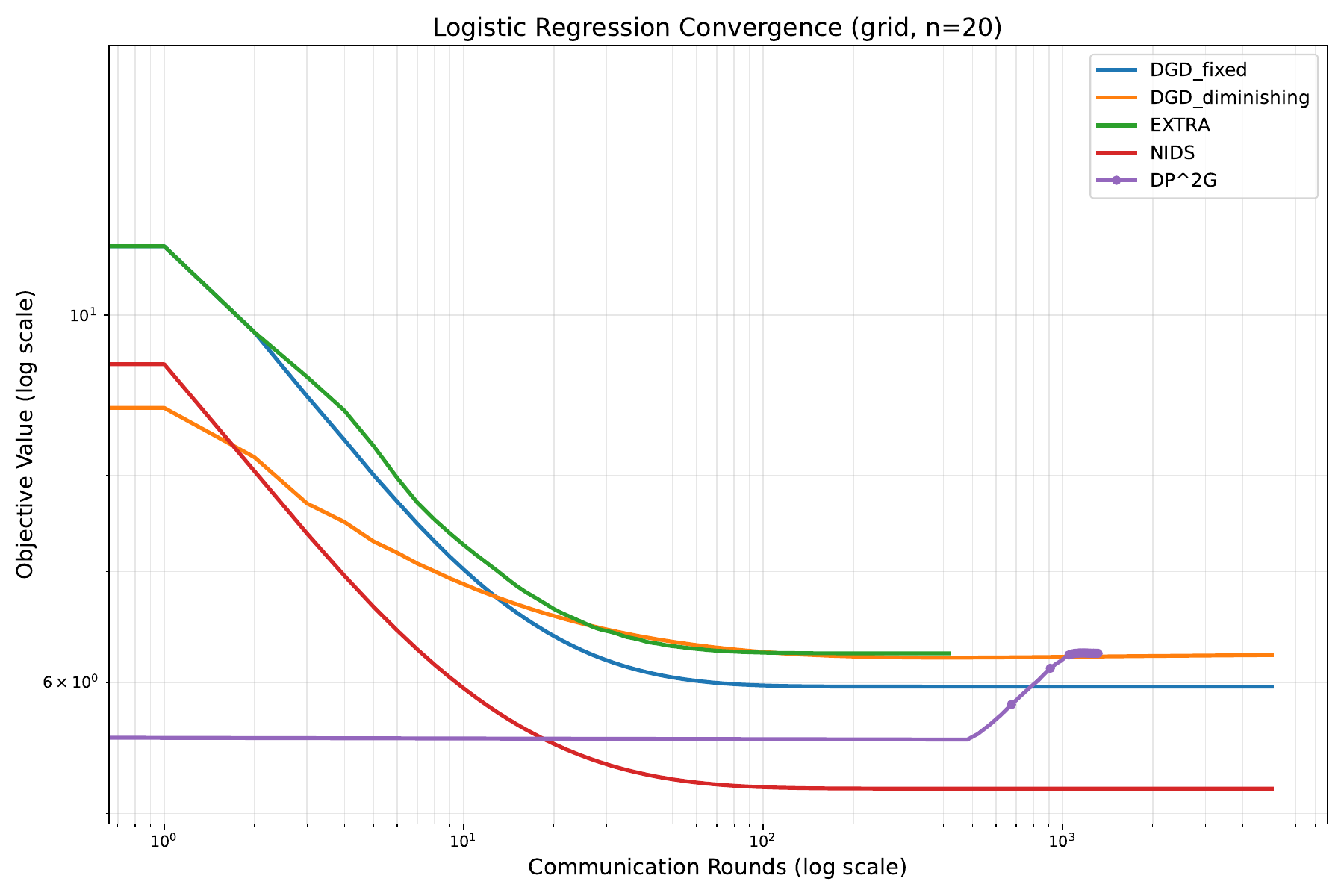}\\[0.6em]
  \includegraphics[width=.85\textwidth]{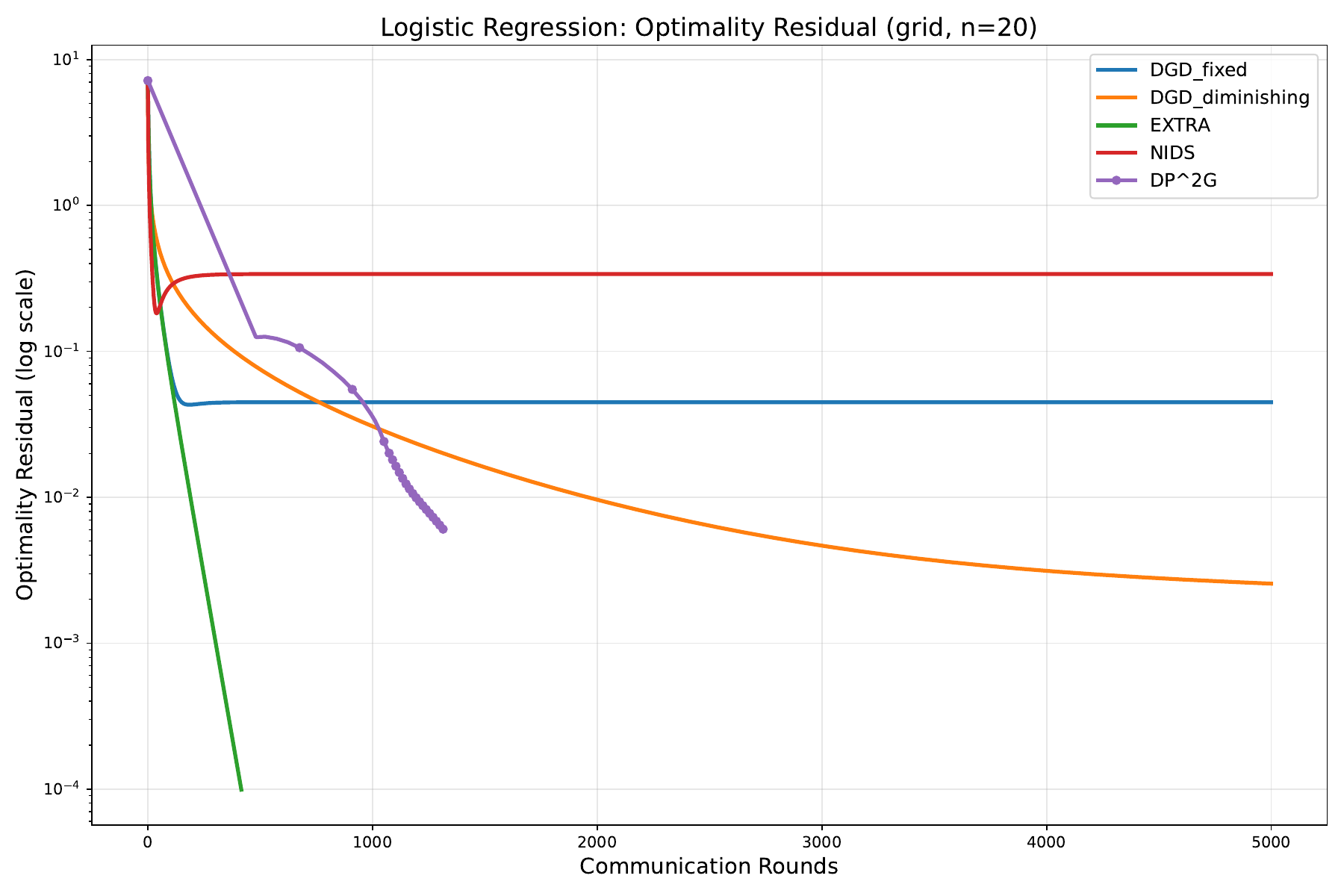}
  \caption{Logistic regression on the $4\times5$ grid: objective residual (top) and optimality residual (bottom). The milder connectivity narrows the gap between DP$^2$G and EXTRA while preserving the memory advantage.}
  \label{fig:logistic_grid}
\end{figure}

\begin{figure}[t]
  \centering
  \includegraphics[width=.85\textwidth]{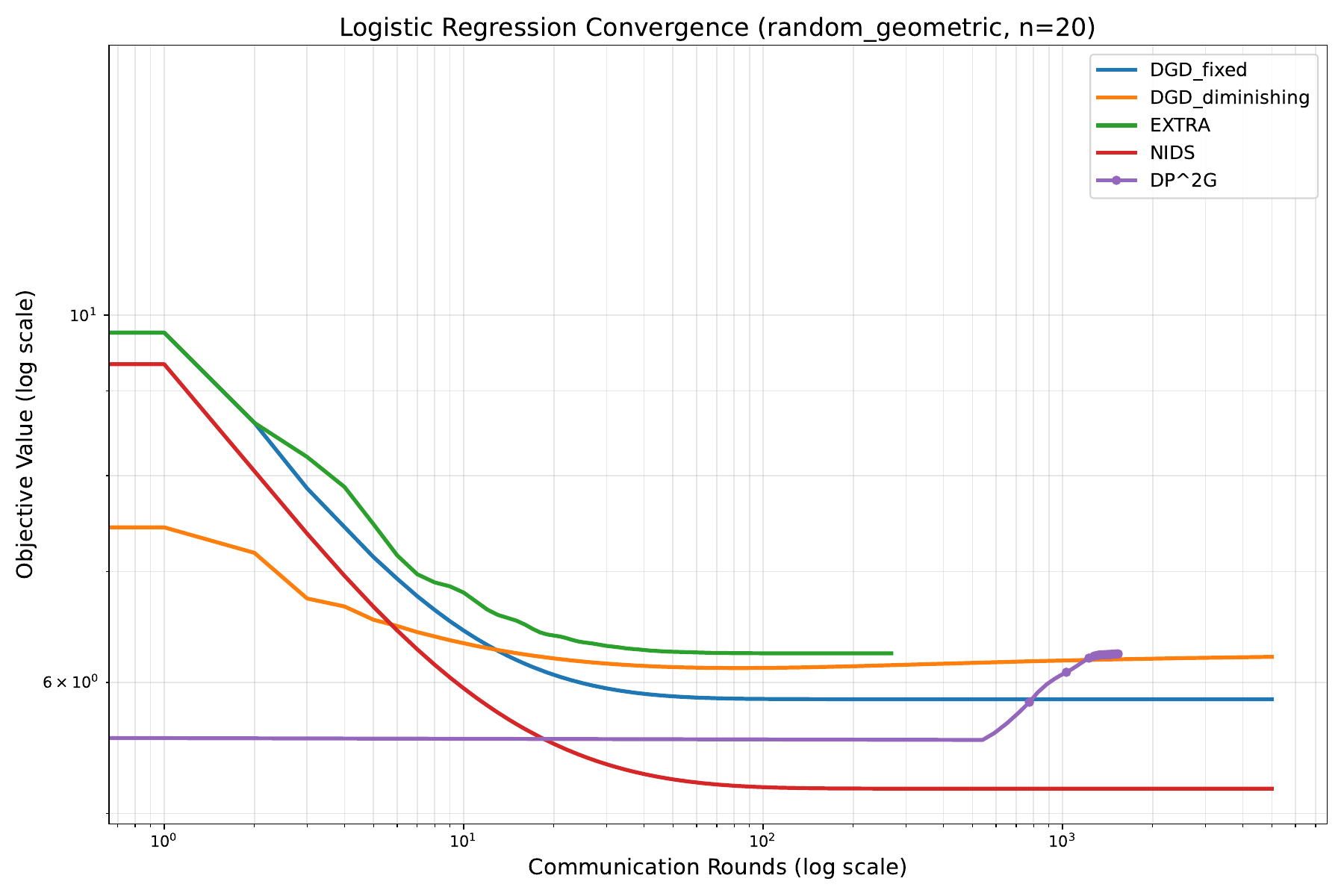}\\[0.6em]
  \includegraphics[width=.85\textwidth]{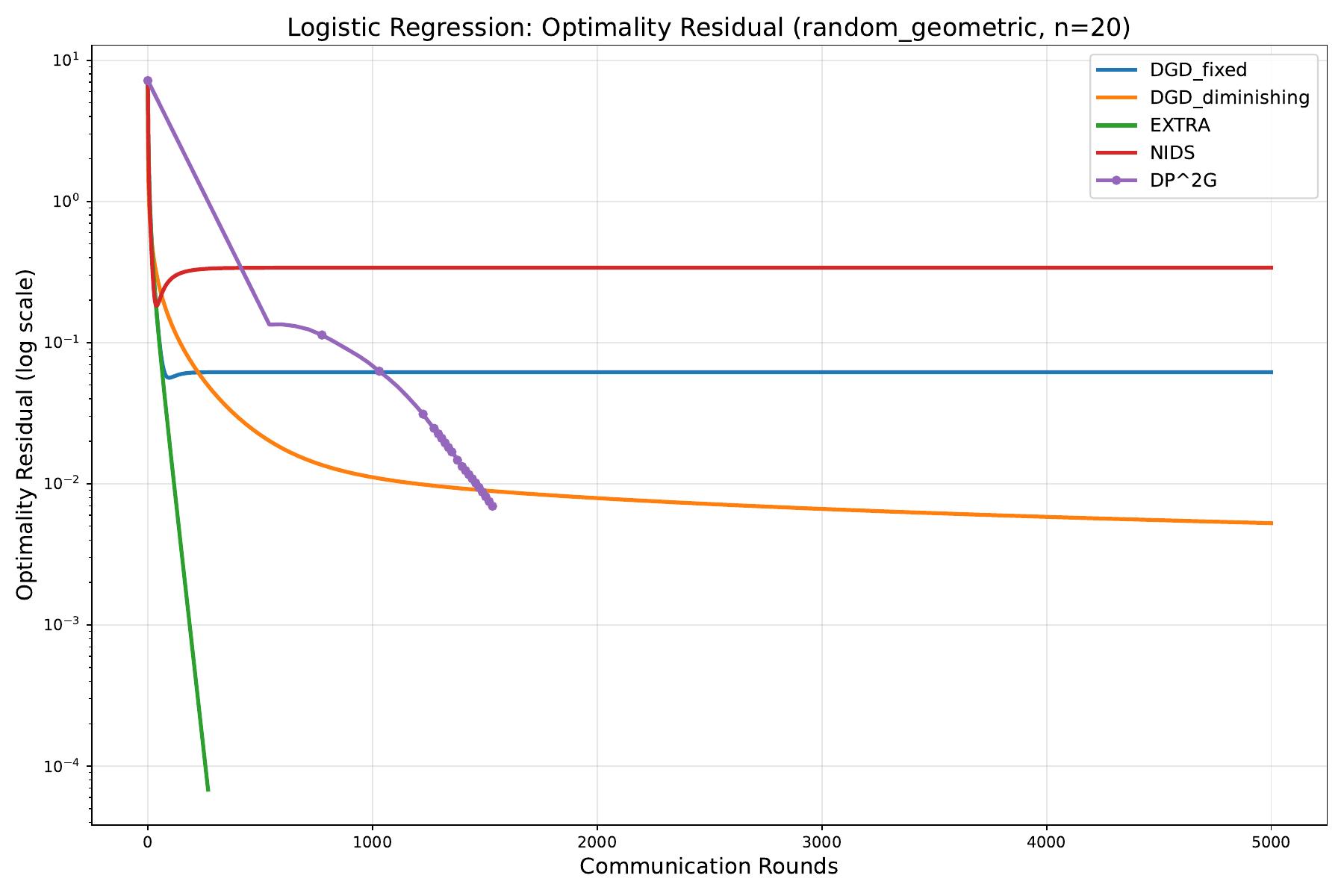}
  \caption{Logistic regression on the random geometric graph: objective residual (top) and optimality residual (bottom). Improved connectivity yields the fastest DP$^2$G decay among the logistic benchmarks.}
  \label{fig:logistic_rg}
\end{figure}

\begin{table}[ht]
  \centering
  \caption{Communication rounds on logistic regression ($n=20$, $m=50$).}
  \label{tab:logistic_comm}
  \begin{tabular}{lccc}
    \toprule
    Algorithm & Ring & $4\times5$ Grid & Random Geometric \\
    \midrule
    DGD (fixed) & $5000^\dagger$ & $5000^\dagger$ & $5000^\dagger$ \\
    DGD (diminishing) & $5000^\dagger$ & $5000^\dagger$ & $5000^\dagger$ \\
    NIDS & $5000^\dagger$ & $5000^\dagger$ & $5000^\dagger$ \\
    EXTRA & 337 & 417 & 269 \\
    DP$^2$G & \textbf{1454} & \textbf{1314} & \textbf{1534} \\
    \bottomrule
  \end{tabular}

  \vspace{0.3em}
  {\footnotesize $^\dagger$Terminated at the 5000-round limit without meeting the stopping rule.}
\end{table}

\subsection{Elastic-net recovery}
The elastic-net benchmark stresses DP$^2$G with a composite nonsmooth objective; no baseline achieved comparable accuracy within the round budget, so we focus on DP$^2$G itself. Figure~\ref{fig:elastic} shows that the method needs $542$ rounds to reduce the objective below $2.82$ and drive the stationarity residual to $3.9\times10^{-5}$. The consensus error reaches $4.4\times10^{-3}$ and the recovered coefficient vector exactly matches the true sparsity pattern (15 nonzeros, $100\%$ precision/recall) with an $\ell_2$ error of $1.0\times10^{-1}$. These results highlight that the penalty continuation handles mixed $\ell_1/\ell_2$ regularization without any algorithmic change.

\begin{figure}[t]
  \centering
  \includegraphics[width=.67\textwidth]{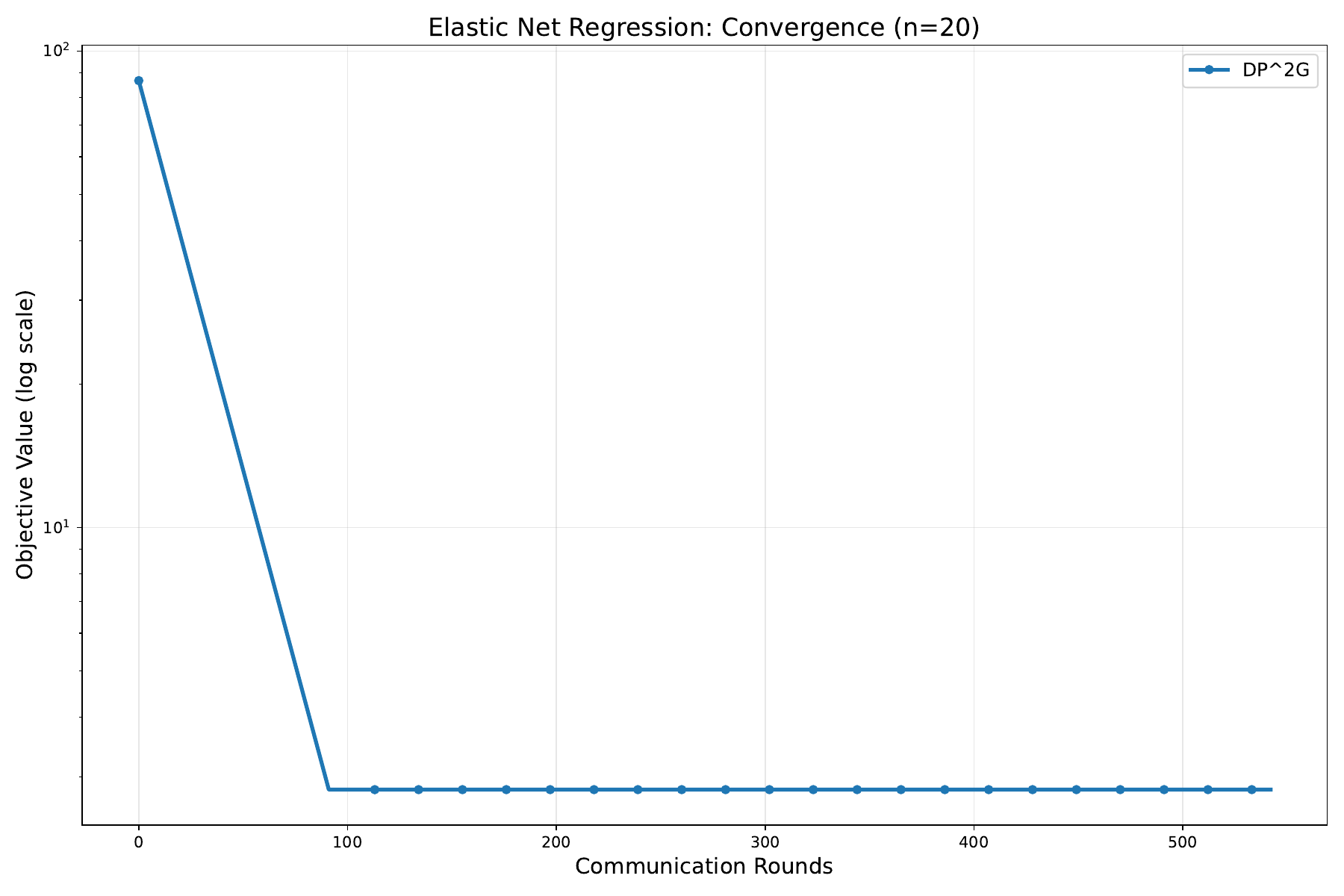}\\[0.6em]
  \includegraphics[width=.67\textwidth]{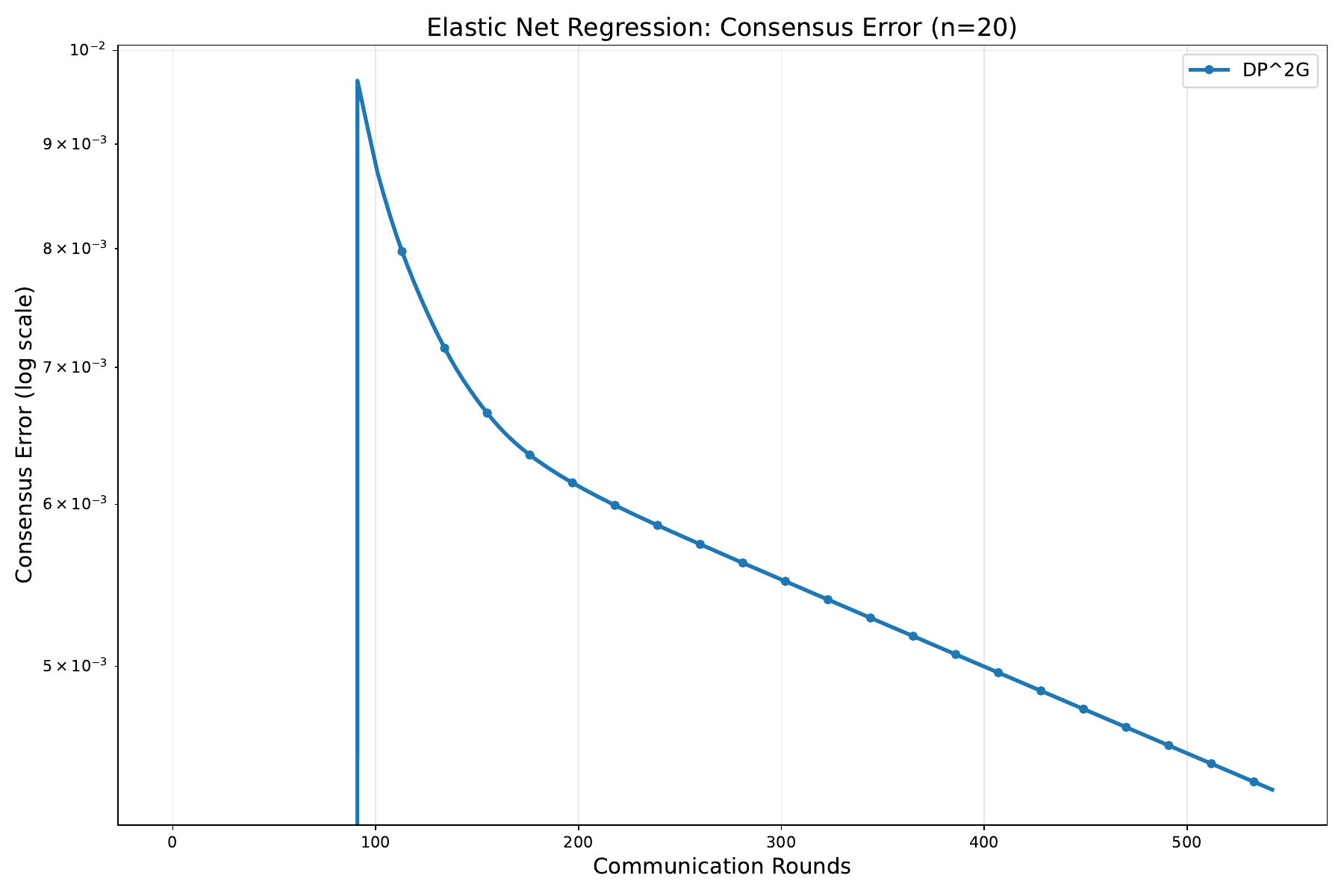}\\[0.6em]
  \includegraphics[width=.8\textwidth]{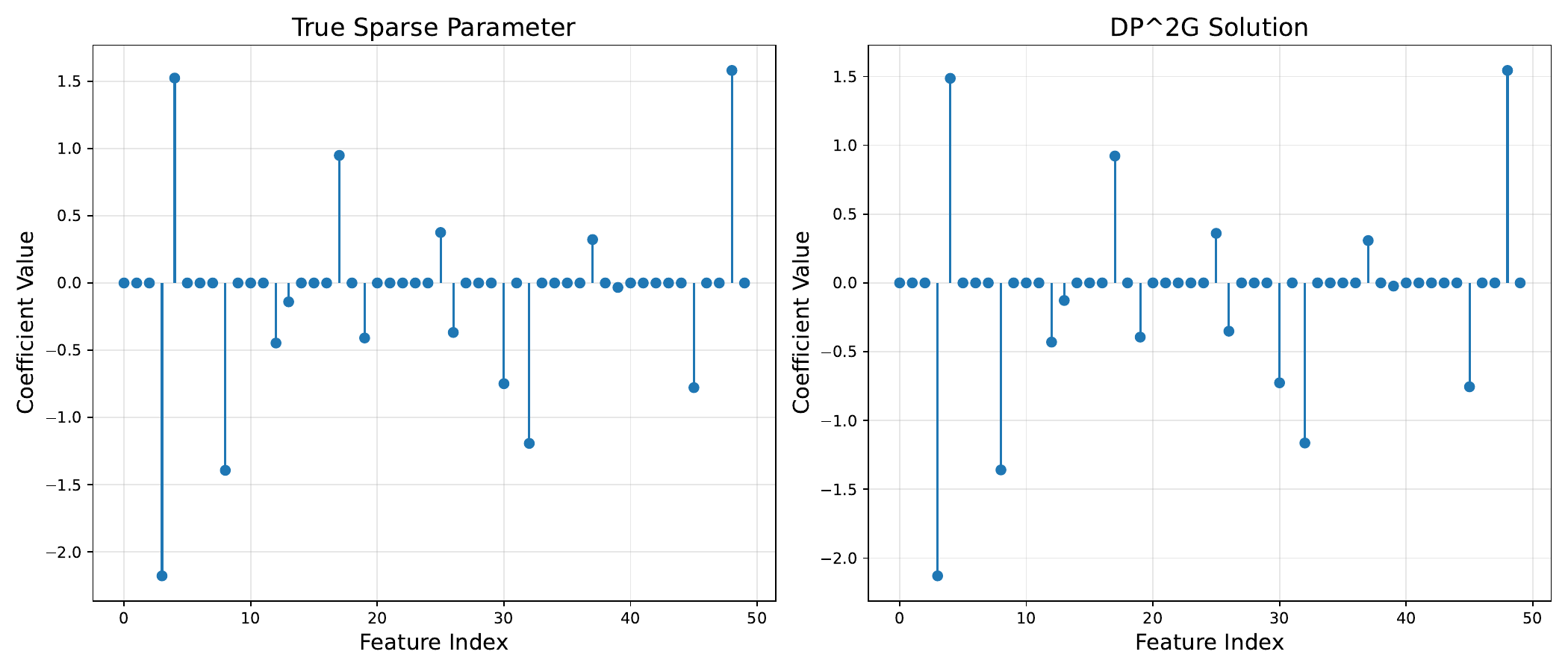}
  \caption{Elastic-net benchmark (random geometric graph, $n=20$): objective residual (top), consensus error (middle), and recovered sparsity pattern (bottom).}
  \label{fig:elastic}
\end{figure}

\subsection{Discussion}
Across all experiments DP$^2$G consistently enforces consensus while keeping the state size close to DGD. On strongly convex problems it trails EXTRA by a small factor yet vastly outperforms one-state baselines that never meet the tolerance within the communication budget. On merely convex models, DP$^2$G demands longer penalty ramps, and practitioners should expect $3$--$4\times$ more rounds than gradient-tracking methods when the graph is poorly connected. Nonetheless, the algorithm remains competitive on grids/RG graphs and seamlessly extends to composite objectives such as the elastic-net without retuning stepsizes or tolerances. These trade-offs validate the appeal of explicit $\ell_1$ penalties when memory is scarce and communication budgets are moderate.

\section{Conclusion}\label{sec:conclusion}
We introduced a modular two-layer framework for decentralized consensus optimization with explicit $\ell_1$ disagreement penalties. The inner layer accepts any saddle-point solver that satisfies a simple interface, while the outer penalty continuation guarantees exactness once the penalty cap is reached. Specializing the inner loop to a primal-dual proximal gradient routine yields the DP$^2$G algorithm, which matches the accuracy of gradient-tracking schemes while keeping only one primal and one dual vector per agent.

The proposed method relies solely on local gradients, neighbor averaging, and component-wise clipping of dual residuals; it therefore preserves the communication footprint of DGD yet enjoys fixed-step convergence guarantees. Our analysis establishes global convergence to consensual critical points, vanishing disagreement along the entire trajectory, and linear rates under strong convexity. Numerical benchmarks on ridge, logistic, and elastic-net problems confirm that the framework delivers strong communication-efficiency trade-offs relative to both one-state baselines and gradient-tracking methods.

Several extensions remain open. Exploring time-varying or directed graphs, asynchronous or event-triggered implementations, and richer composite objectives with local nonsmooth terms may broaden the applicability of the framework. The penalty-based viewpoint also invites alternative inner solvers and adaptive penalty schedules that could further enhance robustness in unreliable networks.

\subsection*{Acknowledgments}
This work was supported in part by National Natural Science Foundation of China (grant no. 12201428) 
and in part by the Natural Science Foundation of Top Talent of SZTU (grant no. GDRC202136).

\end{document}